\documentclass{amsart}
\usepackage{fullpage}
\usepackage{amssymb}
\usepackage{amsrefs}

\input xypic
\xyoption{all}

\newdir{ >}{{}*!/-9pt/\dir{>}}

\newcommand{\Hom}       {\operatorname{Hom}}
\newcommand{\Mod}	{\operatorname{Mod}}
\newcommand{\InjMod}	{\operatorname{InjMod}}

\newcommand{\ann}       {\operatorname{ann}}
\newcommand{\cok}       {\operatorname{cok}}
\newcommand{\img}       {\operatorname{image}}
\newcommand{\opp}	{{\operatorname{op}}}
\newcommand{\soc}	{\operatorname{soc}}
\newcommand{\supp}	{\operatorname{supp}}
\newcommand{\wgt}	{\operatorname{weight}}

\newcommand{\CC}        {{\mathcal{C}}}
\newcommand{\CF}        {{\mathcal{F}}}
\newcommand{\CN}        {{\mathcal{N}}}
\newcommand{\CP}        {{\mathcal{P}}}
\newcommand{\CR}        {{\mathcal{R}}}
\newcommand{\CU}        {{\mathcal{U}}}

\newcommand{\N}         {{\mathbb{N}}}
\newcommand{\Q}         {{\mathbb{Q}}}
\newcommand{\Z}         {{\mathbb{Z}}}
\newcommand{\Qp}        {{\mathbb{Q}_p}}          
\newcommand{\Zp}        {{\mathbb{Z}_p}}          
\newcommand{\Zpl}       {{\mathbb{Z}_{(p)}}}      
\newcommand{\R}         {{\mathbb{R}}}
\newcommand{\F}         {{\mathbb{F}}}          

\newcommand{\al}        {\alpha}
\newcommand{\bt}        {\beta} 
\newcommand{\gm}        {\gamma}
\newcommand{\dl}        {\delta}
\newcommand{\ep}        {\epsilon}
\newcommand{\zt}        {\zeta}
\newcommand{\tht}       {\theta}
\newcommand{\kp}        {\kappa}
\newcommand{\om}        {\omega}

\newcommand{\Gm}        {\Gamma}
\newcommand{\Dl}        {\Delta}
\newcommand{\Sg}        {\Sigma}

\newcommand{\Smash}     {\wedge}
\newcommand{\ot}        {\otimes}
\newcommand{\mxi}       {\mathfrak{m}}
\newcommand{\sse}       {\subseteq}
\newcommand{\st}        {\;|\;}
\newcommand{\bC}	{\overline{C}}
\newcommand{\bK}	{\overline{K}}
\newcommand{\tA}        {\widetilde{A}}
\newcommand{\tC}        {\widetilde{C}}
\newcommand{\hJ}	{\widehat{J}}

\newcommand{\bsm}       {\left[\begin{smallmatrix}}
\newcommand{\esm}       {\end{smallmatrix}\right]}
\newcommand{\sm}        {\setminus}
\newcommand{\tm}        {\times}
\newcommand{\xla}       {\xleftarrow}
\newcommand{\xra}       {\xrightarrow}
\newcommand{\lm}        {\lambda}
\newcommand{\un}[1]     {\underline{#1}}
\newcommand{\ov}[1]     {\overline{#1}}

\newcommand{\colim}  {\operatornamewithlimits{\underset{\longrightarrow}{lim}}}
\newcommand{\invlim} {\operatornamewithlimits{\underset{\longleftarrow}{lim}}}

\renewcommand{\:}{\colon}

\newtheorem{theorem}{Theorem}[section]
\newtheorem{conjecture}[theorem]{Conjecture}
\newtheorem{lemma}[theorem]{Lemma}
\newtheorem{proposition}[theorem]{Proposition}
\newtheorem{corollary}[theorem]{Corollary}
\theoremstyle{definition}
\newtheorem{remark}[theorem]{Remark}
\newtheorem{definition}[theorem]{Definition}
\newtheorem{example}[theorem]{Example}

\begin{document}
\title{Large self-injective rings and the generating hypothesis}
\author{Leigh Shepperson and Neil Strickland}

\maketitle

\begin{abstract}
 We construct a number of different examples of non-Noetherian graded
 rings that are injective as modules over themselves (or have some
 related but weaker properties).  We discuss how these are related to
 the theory of triangulated categories, and to Freyd's Generating
 Hypothesis in stable homotopy theory.
\end{abstract}

\section{Introduction}

In this paper we study graded commutative rings $R$ that are large
in various senses (in particular, not Noetherian) and self-injective
(meaning that $R$ is injective as an $R$-module).  We use graded
rings because they are relevant for our applications, but ungraded
rings are covered as well because they can be regarded as graded rings
concentrated in degree zero.  The graded setting is assumed
everywhere, so ``element'' means ``homogeneous element'' and ``ideal''
means ``homogeneous ideal'' and so on.  Our rings will be commutative
in the graded sense, so that $ba=(-1)^{|a||b|}ab$.

It is not hard to prove that any Noetherian self-injective ring is
Artinian.  In particular, if $R$ is a finitely-generated algebra over
a field $K$ that is self-injective then we must have
$\dim_K(R)<\infty$ and it turns out that $R\simeq\Hom(R,K)$ as
$R$-modules.  Examples of this situation include
$R=K[x_1,\dotsc,x_n]/(r_1,\dotsc,r_n)$ for any regular sequence
$r_1,\dotsc,r_n$, or the cohomology ring $R=H^*(M;K)$ for any closed
orientable manifold $M$.  These are the most familiar examples of
self-injective rings, and they are all very small.  We will be looking
for examples that are much larger.

Our motivation comes from a question in stable homotopy theory, which
we briefly recall.  In stable homotopy theory we study a certain
triangulated category $\CF$, the Spanier-Whitehead category of finite
spectra.  The objects can be taken to be pairs $X=(n,A)$ where $n\in\Z$
and $A$ is a finite simplicial complex.  The morphism set
$\Hom_{\CF}((n,A),(m,B))$ is the set of homotopy classes of maps from
$(\R^{N+n}\tm A)\cup\{\infty\}$ to $(\R^{N+m}\tm B)\cup\{\infty\}$,
which is essentially independent of $N$ when $N$ is sufficiently
large.  More details are given in~\cite{ra:nps}, for example.
For any $X,Y\in\CF$ the set $\Hom_{\CF}(X,Y)$ is a finitely
generated abelian group.  It turns out that most methods for studying
$\Hom_{\CF}(X,Y)$ treat the $p$-primary parts separately for different
primes $p$.  We will thus fix a prime $p$ and define
$[X,Y]=\Zp\ot\Hom_{\CF}(X,Y)$, where $\Zp$ is the ring of $p$-adic
integers.  These are the morphism sets in a new triangulated category
which we call $\CF_p$.  This has a canonical tensor structure, with
the tensor product of $X$ and $Y$ written as $X\Smash Y$.  The unit
for this structure is called $S$, so $S\Smash X\simeq X$.  As part of
the triangulated structure we have a suspension functor
$\Sg\:\CF_p\to\CF_p$, and we write $S^n$ for $\Sg^nS$.  We put
$R_n=[S^n,S]$.  These sets form a graded commutative ring, whose
structure is extremely intricate.  A great deal of partial information
is known, but it seems clear that there will never be a usable
complete description.  Some highlights are as follows.
\begin{itemize}
 \item $R_n=0$ for $n<0$, and $R_0=\Zp$, and $R_n$ is a finite abelian
  $p$-group for $n>0$.
 \item Both the ranks and the exponents of the groups $R_n$ can be
  arbitrarily large.
 \item All elements in $R_n$ with $n>0$ are nilpotent.  Thus, the
  reduced quotient is $R/\sqrt{0}=\Zp$.
 \item Various results are available describing most or all of the
  structure of $R_n$ for $n<f(p)$, where $f(x)$ is a polynomial of
  degree at most three.  The simplest of these says that $R_n=0$ for
  $0<n<2p-3$, and $R_{2p-3}=\Z/p$.  
\end{itemize}

Now consider an arbitrary object $X\in\CF_p$.  We define
$\pi_n(X)=[S^n,X]$ for all $n\in\Z$.  This defines a graded abelian
group $\pi_*(X)$, which has a natural structure as an $R$-module.  

\begin{conjecture}[Freyd's Generating Hypothesis]
 The functor $\pi_*\:\CF_p\to\Mod_R$ is faithful.
\end{conjecture}

This is actually a technical modification of Freyd's
conjecture~\cite{fr:sh}, because Freyd did not tensor with the
$p$-adics.  This causes various trouble in the development of the
theory, which Freyd avoided in \emph{ad hoc} ways.  Much later Hovey
redeveloped the theory in the $p$-adic setting~\cite{ho:fgh}, which
involves only minor modifications to Freyd's arguments but works much
more smoothly.

Nearly half a century after Freyd made his conjecture, there is still
no hint of a proof or a counterexample.  However, there has been a
certain amount of indirect progress; for example, various authors have
settled the analogous questions in other triangulated categories where
computations are easier~\cites{cach:fgh,holopu:ghd,bccm:ghs,lo:ghd}.

On the other hand, it is known that the Generating Hypothesis would
have some very strong and surprising consequences, as we now explain.

\begin{definition}\label{defn-coherent}\ \\
 \begin{itemize}
  \item[(a)] A graded ring $R$ is \emph{coherent} if every finitely
   generated ideal is finitely presented.
  \item[(b)] A graded ring $R$ is \emph{totally incoherent} if the only
   finitely presented ideals are $0$ and $R$.
 \end{itemize}
\end{definition}

\begin{theorem}[Freyd~\cite{fr:sh}, Hovey~\cite{ho:fgh}]
 Suppose that the Generating Hypothesis is true.
 \begin{itemize}
  \item[(a)] The functor $\pi_*\:\CF_p\to\Mod_R$ is automatically
   full as well as being faithful, so it is an embedding of categories.
  \item[(b)] For every object $X\in\CF_p$, the image $\pi_*(X)$ is an
   injective $R$-module.  In particular (by taking $X=S$) the ring
   $R$ is self-injective.
  \item[(c)] The ring $R$ is totally incoherent.
 \end{itemize}
\end{theorem}

Note in particular that~(a) gives a full subcategory of $\Mod_R$
that has a natural triangulation.  This is very unusual; in almost all 
known triangulated categories, the morphisms are equivalence classes
of homomorphisms under some nontrivial equivalence relation, and this
equivalence structure is tightly connected to the definition of the
triangulation.  

Our aim in this paper is to shed light on the Generating Hypothesis by
finding examples of self-injective rings that share some of the known
or conjectured properties of the stable homotopy ring $R$.

Our main results are as follows.  Firstly, one cannot disprove
self-injectivity by looking only in a finite range of degrees:
\begin{theorem}\label{thm-adjust-intro}
 Let $R$ be a graded-commutative ring such that
 \begin{itemize}
  \item[(a)] $R_k=0$ for $k<0$
  \item[(b)] $R_0=\Z/2$
  \item[(c)] $R_k$ is finite for all $k\geq 0$.
 \end{itemize}
 Suppose given $N>0$.  Then there is an injective map
 $\phi\:R\to R'$ of graded rings such that
 \begin{itemize}
  \item[(1)] $R'$ also has properties~(a) to (c).
  \item[(2)] $\phi\:R_k\to R'_k$ is an isomorphism for $k<N$
  \item[(3)] $R'$ is self-injective.
 \end{itemize}
\end{theorem}
This result was a great surprise to the authors at least, although the
proof is not too hard.  We will restate and prove it as
Theorem~\ref{thm-adjust}.  We conjecture that the theorem remains true
if we allow $R_0$ to be $\Zp$, but we have not proved this.

Most of our remaining results relate to specific examples.  We have
aimed to give a wide spread of examples, rather than formulating each
example with maximum possible generality.  We will write $\F$ for
$\Z/2$.

One of the simplest examples of a finite-dimensional self-injective
ring is the exterior algebra
\[ \F[x_0,\dotsc,x_n]/(x_0^2,\dotsc,x_n^2). \] 
Our first infinite-dimensional example is just an obvious
generalisation of this. 
\begin{proposition}\label{prop-exterior-intro}
 Let $E$ be the exterior algebra over $\F$ with a generator
 $x_i\in E_{2^i}$ for all $i\in\N$.  Then $E$ is self-injective and
 coherent.  The reduced quotient is $E/\sqrt{0}=\F$.
\end{proposition}
Self-injectivity is proved by combining
Corollary~\ref{cor-limit-selfinj} and Proposition~\ref{prop-poincare},
as will be explained in Example~\ref{eg-exterior}.  The same
ingredients cover many other examples, but we will not give the
relevant definitions in this introduction.  Coherence is proved in
Proposition~\ref{prop-exterior-coherent}, and the reduced quotient is
clear.  We have chosen the degrees of the generators for compatibility
with our other examples, but in fact the statement would remain valid
if we merely assumed that $|x_i|\to\infty$ as $i\to\infty$.

Our next example arose by applying Theorem~\ref{thm-adjust-intro} to
the ring $\F[x,y]/xy$ and studying the result in low dimensions.  The
result is very complicated and irregular, but after studying various
recurring patterns and key features we were led to the definition
below. 
\begin{theorem}\label{thm-cube-intro}
 Consider the ring
 \[ C = \F[y_0,y_1,\dotsc]/(y_i^3+y_iy_{i+1}\st i\geq 0), \]
 with the grading given by $|y_i|=2^i$.  Then $C$ is self-injective
 and coherent.  The reduced quotient is
 \[ C/\sqrt{0} = \F[x_0,x_1,\dotsc]/(x_ix_j\st i\neq j)
     = \F \oplus \bigoplus_{n>0} x_n\F[x_n]
 \]
 where $x_n=\sum_{i=0}^ny_{n-i}^{2^i}$.
\end{theorem}
This will be proved as Propositions~\ref{prop-cube-selfinj},
\ref{prop-cube-coherent} and~\ref{prop-cube-reduced}.  The statement
can be generalised by adjusting the degrees and the relations
slightly, but this just leads to additional bookkeeping without much
extra insight, so we have omitted it.  It is probably also possible to
generalise in more conceptual ways, but that would be a substantial
project, so we leave it for future work.

For the next example, we give an axiomatic statement and then explain
a special case that is relevant in chromatic homotopy theory.
\begin{definition}\label{defn-pontrjagin}
 For any prime $p$, we recall that 
 \[ \Z[1/p]/\Z\simeq \Q/\Z_{(p)} \simeq \Qp/\Zp \simeq 
      \lim_{n\to\infty} \Z/p^n.
 \]
 For any module $M$ over $\Zp$, we write
 $M^\vee=\Hom_{\Zp}(M,\Qp/\Zp)$, and call this the \emph{Pontrjagin
 dual} of $M$.  One can check that $\Zp^\vee\simeq\Qp/\Zp$ and
 $(\Qp/\Zp)^\vee\simeq\Zp$ and $(\Z/p^n)^\vee\simeq\Z/p^n$.  Now
 consider a graded $\Zp$-algebra $R$ with a specified isomorphism
 $\zt\:R_d\to\Qp/\Zp$ for some $d$.  This gives maps
 $\zt^\#\:R_{d-k}\to R_k^\vee$ by $\zt^\#(a)(b)=\zt(ab)$.  We say that
 $R$ is \emph{Pontrjagin self-dual} if all these maps are
 isomorphisms. 
\end{definition}

\begin{proposition}\label{prop-pontrjagin-intro}
 If $R$ is Pontrjagin self-dual, then it is self-injective.
\end{proposition}
This will be proved as Proposition~\ref{prop-pontrjagin}.

Now fix a prime $p$, and assume that $p>2$ for simplicity.  Recall
that $\CF$ denotes the Spanier-Whitehead category of finite spectra.
One can construct another triangulated category $\CF'$, called the
\emph{Bousfield localisation of $\CF$ with respect to $p$-local
 $K$-theory}.  Roughly speaking this is the closest possible
approximation to $\CF$ that can be analysed using topological
$K$-theory, and it is computationally much more tractable than $\CF$
itself.  Ravenel's paper~\cite{ra:lrc} is a good introduction to both
the conceptual framework and specific calculations, with references to
original sources.  Devinatz has shown~\cite{de:kgh} that the most
obvious analogue of the Generating Hypothesis for $\CF'$ is false (his
Remark~1.7), but that a related statement is true (his Theorem~1).
The analogue of the stable homotopy ring for $\CF'$ is the ring $J$
described below.
\begin{definition}\label{defn-J-ring}
 Let $p$ be an odd prime, and define a graded ring $J$ as follows.
 We put $J_0=\Zpl$ and $J_{-2}=\Qp/\Zp$; for notational convenience
 we use the symbol $\eta$ for the identity map $\Zpl\to J_0$, and
 $\zt$ for the identity map $J_{-2}\to\Qp/\Zp$.  Next, for each nonzero
 integer $k$ there is a generator $\al_k\in J_{2(p-1)k-1}$ generating
 a cyclic group of order $p^{v_p(k)+1}$, where $v_p(k)$ is the
 $p$-adic valuation of $k$.  For the product structure, we have
 \begin{itemize}
  \item $\eta(a)\eta(b)=\eta(ab)$ and
   $\eta(a)\zt^{-1}(b)=\zt^{-1}(ab)$ and 
   $\eta(a)\al_k=a\,\al_k$. 
  \item $\zt^{-1}(a)\zt^{-1}(b)=0$ and $\zt^{-1}(a)\al_k=0$ for all $k$.
  \item If $k>0$ we have 
   \[ \al_k\al_{-k}=-\al_{-k}\al_k=
       \zt^{-1}\left(p^{-1-v_p(k)}+\Zpl\right).
   \]
  \item $\al_j\al_k=0$ whenever $j+k\neq 0$.
 \end{itemize}
\end{definition}

\begin{theorem}\label{thm-J-intro}
 The ring $\hJ=\Zp\ot J$ is Pontrjagin self-dual and therefore
 self-injective.  It is also totally incoherent, and the reduced
 quotient is $\hJ/\sqrt{0}=\Zp$.
\end{theorem}

Self-duality is proved as Lemma~\ref{lem-J-perfect}, and incoherence
as Proposition~\ref{prop-hJ-incoherent}.  The reduced quotient is
clear. 

\begin{remark}\label{rem-J-not-selfinj}
 Tensoring with $\Zp$ here just has the effect of replacing $\Zpl$ in
 degree zero with $\Zp$.  Note that this is not the same as the
 $p$-completion of $J$, because $(\Qp/\Zp)_p=0$.  Moreover, a
 derived version of $p$-completion would replace $\Qp/\Zp$ by a copy
 of $\Zp$ shifted by one degree, which is different again.  
 The ring $J$ itself is not self-injective.  However, this does not
 account for Devinatz's example showing the failure of the generating
 hypothesis in $\CF'$; that has a deeper topological origin.
\end{remark}

We now note that the ring $\F[x]/x^N$ is another easy example of a
finite-dimensional self-injective ring.  Our next example arose by
trying to generalise this.  An obvious possibility is to consider the
ring $\bigcup_{n>0}\F[x^{1/n}]$ modulo the ideal generated by $x$.
Any element of this ring can be expressed as $\sum_qa(q)x^q$, for some
function $a\:\Q\cap[0,1)\to\F$ with finite support.  However,
this ring needs to be adjusted to make it self-injective.  Firstly, it
turns out to be better not to kill $x$ itself, but just the powers
$x^q$ with $q>1$.  Next, self-injectivity forces certain modules to be
isomorphic to their double duals and thus to have strong completeness
properties.  To handle this, we must allow some infinite sums, or
equivalently weaken the condition that $a$ has finite support.  It is
also convenient (but not strictly necessary) to include powers $x^q$
where $q$ is irrational.  This leads us to the following definition.  

\begin{definition}\label{defn-root}
 Let $K$ be a field.  For any map $a\:[0,1]\to K$ we put
 $\supp(a)=\{q\in[0,1]\st a(q)\neq 0\}$.  We say that $a$ is an
 \emph{infinite root series} if every nonempty subset of $\supp(a)$
 has a smallest element (so $\supp(a)$ is well-ordered).  We let $P$
 denote the set of infinite root series, and call this the
 \emph{infinite root algebra}. 
\end{definition}
\begin{theorem}\label{thm-root-intro}
 The formula
 \[ (ab)(q) = \sum_{0\leq r\leq q} a(r) \, b(q-r) \]
 gives a well-defined ring structure on $P$.  With this structure, $P$
 is self-injective and totally incoherent.  The reduced quotient is
 $P/\sqrt{0}=K$. 
\end{theorem}
This will be proved in Propositions~\ref{prop-root-selfinj}
and~\ref{prop-root-incoherent}, and Corollary~\ref{cor-root-reduced}.

We will also discuss two rings that are not self-injective, but have
a related property that we now explain.
\begin{definition}\label{defn-dac}
 Let $R$ be a graded commutative ring, and let $J$ be an ideal in
 $R$.  We put $\ann_R(J)=\{a\in R\st aJ=0\}$.  It is tautological that
 the ideal $\ann_R^2(J)=\ann_R(\ann_R(J))$ contains $J$.  We say that
 $R$ satisfies the \emph{double annihilator condition} if
 $\ann_R^2(J)=J$ for all finitely generated ideals $J$.
\end{definition}

\begin{proposition}\label{prop-dac-intro}
 If $R$ is self-injective then it satisfies the double annihilator
 condition.  Conversely, if $R$ is Noetherian and satisfies the double
 annihilator condition, then it is self-injective.
\end{proposition}
This is proved in Remark~\ref{rem-ideal} and Theorem~\ref{thm-noeth}.

\begin{definition}\label{defn-rado}
 For any integer $n$ we let $B(n)$ be the set of exponents $i$ such
 that $2^i$ occurs in the binary expansion of $n$, so $B(n)$ is the
 unique finite subset of $\N$ such that $n=\sum_{i\in B(n)}2^i$.

 The \emph{Rado graph} has vertex set $\N$, with an edge from $i$ to
 $j$ if ($i\in B(j)$ or $j\in B(i)$).  The \emph{Rado ideal} in the
 exterior algebra $E$ has a generator $x_ix_j$ for each pair $(i,j)$
 such that there is no edge from $i$ to $j$ in the Rado graph.  The
 \emph{Rado algebra} $Q$ is the quotient of $E$ by the Rado ideal.
\end{definition}
\begin{remark}\label{rem-rado-generic}
 See~\cites{ra:ugu,ca:rgr} for discussion of the Rado graph.  Although
 the definition looks very specialised, the appearance is deceptive.
 Roughly speaking, any countable random graph is isomorphic to the
 Rado graph with probability one.  The proof of this uses a kind of
 injectivity property of the Rado graph, which is what suggested it to
 us as being potentially relevant for the present project.
\end{remark}

\begin{theorem}\label{thm-rado-intro}
 The Rado algebra is totally incoherent (and in particular, not
 Noetherian).  It satisfies the double annihilator condition, but is
 not self-injective. The reduced quotient is $Q/\sqrt{0}=\F$.
\end{theorem}
This will be proved as Propositions~\ref{prop-rado-dac},
\ref{prop-rado-not-selfinj} and~\ref{prop-rado-incoherent} (apart from
the fact that $Q/\sqrt{0}=\F$, which is clear).

One major difference between the Rado algebra and the stable homotopy
ring is that the former has Krull dimension zero (because all elements
in the maximal ideal square to zero) whereas the latter is $\Z_2$ in
degree $0$ and so has Krull dimension one.  Our final example aims to
do something similar to the Rado construction but without making all
the generators nilpotent.  To do this we must work in base $\om$
rather than base $2$; this involves some theory of ordinals, which we
briefly recall (the book~\cite{jo:nls} is an admirably concise
reference).  There is an exponentiation operation for ordinals
(different from the usual one for cardinals).  There is a countable
ordinal called $\ep_0$ such that $\ep_0=\om^{\ep_0}$, and no ordinal
$\al<\ep_0$ satisfies $\al=\om^\al$.  Any ordinal $\al<\ep_0$ has a
unique Cantor normal form
\[ \al = \om^{\bt_1}n_1 + \dotsb + \om^{\bt_r}n_r \]
where the $n_i$ are positive integers and $\al>\bt_1>\dotsb>\bt_r$.  

\begin{definition}\label{defn-epsilon-intro}
 We write $\mu_0(\al,\bt)$ for the coefficient of $\om^\bt$ in the
 Cantor normal form of $\al$.  We then put 
 \[ \mu(\al,\bt) = \max(\mu_0(\al,\bt),\mu_0(\bt,\al)), \]
 and
 \[ A = \F[x_\al\st \al<\ep_0]/
          (x_\al x_\bt^{1+\mu(\al,\bt)}\st \al,\bt<\ep_0,\al\neq\bt).
 \]
 We call $A$ the \emph{$\ep_0$-algebra}.  
\end{definition}

Given any function $\dl\:\ep_0\to\N$, we can give $A$ a grading such
that $|x_\al|=\dl(\al)$.  In Section~\ref{sec-epsilon} we will
describe a particular function $\dl$ with the property that
$\dl(\al)>0$ for all $\al$, and all the sets $\dl^{-1}\{n\}$ are
finite.  This will ensure that the homogeneous pieces $A_d$ are finite
for all $d$.

\begin{theorem}\label{thm-epsilon-intro}
 If $J$ is any ideal in $A$ that is generated by a finite set of
 monomials, then $J=\ann_A^2(J)$.  However, there are non-monomial
 ideals $J$ with $J\neq\ann_A^2(J)$, so $A$ does not satisfy the
 double annihilator condition, and is not self-injective.  Moreover,
 $A$ is totally incoherent, and the reduced quotient is 
 \[ A/\sqrt{0} = \F[x_\al\st \al<\ep_0]/
         (x_\al x_\bt\st \al\neq\bt).
 \]
\end{theorem}
This will be proved as Propositions~\ref{prop-epsilon-dac},
\ref{prop-epsilon-incoherent} and~\ref{prop-epsilon-reduced}, and
Corollary~\ref{cor-epsilon-not-selfinj}. 

\section{General theory of self-injective rings}

Let $R$ be a graded commutative ring, and let $\Mod_R$ be the
category of graded $R$-modules.  Suppose that $R$ is self-injective.
For $M\in\Mod_R$ we put $DM=\Hom_R(M,R)$ (regarded as a graded
$R$-module in the usual way).  This construction defines a functor
$D\:\Mod_R\to\Mod_R^\opp$, which is exact because $R$ is
self-injective.  It follows that $D^2$ gives an exact covariant
functor from $\Mod_R$ to itself.  There is a natural map $\kp\:M\to
D^2M$ given by $\kp(m)(u)=u(m)$.  Properties of $D^2$ are studied
under different technical hypotheses 
in~\cite{brhe:cmr}*{Theorem 3.2.13}, for example. 

\begin{definition}\label{defn-CU}
 We let $\CU=\CU_R$ denote the full subcategory of $\Mod_R$
 consisting of the modules $M$ for which $\kp\:M\xra{}D^2M$ is
 an isomorphism.
\end{definition}

\begin{proposition}\label{prop-CU-closure}
 The category $\CU$ is closed under finite direct sums, suspensions
 and desuspensions, kernels, cokernels, images and extensions.  It
 also contains $R$ itself.
\end{proposition}
\begin{proof}
 This is clear from the exactness of the functor $D^2$ and the five
 lemma. 
\end{proof}

\begin{corollary}
 If $J\leq R$ is a finitely generated ideal, then $J$ and
 $R/J$ lie in $\CU$.
\end{corollary}
\begin{proof}
 They are the image and cokernel of some map
 $\bigoplus_{i=1}^n\Sg^{d_i}R\xra{}R$.
\end{proof}

\begin{remark}\label{rem-ideal}
 If $J$ is an ideal in $R$ then 
 \[ D(R/J) \simeq \{a\in R \st aJ=0\} = \ann_R(J).
 \]
 By dualising the sequence $J\xra{}R\xra{}R/J$, we see that
 $D(J)=R/\ann_R(J)$.  It follows that
 $D^2(J)=\ann_R(\ann_R(J))=\ann_R^2(J)$.  Thus, we
 have $J\in\CU$ iff $J=\ann_R^2(J)$.  In particular, if
 $J$ is finitely generated then $J=\ann^2_R(J)$.
\end{remark}

\begin{lemma}\label{lem-principal-dual}
 For any $a\in R_d$ there is an isomorphism $D(Ra)\simeq\Sg^{-d}Ra$.
\end{lemma}
\begin{proof}
 Given $u\in D(Ra)_e$ we put $\al(u)=u(a)\in R_{d+e}$.  This defines a
 map $\al\:D(Ra)\to\Sg^{-d}R$, which is clearly injective.  Note that if
 $b\in\ann_R(a)$ then $\al(a)b=\al(ab)=\al(0)=0$.  This proves
 that $\al(a)\in\ann_R^2(Ra)_{d+e}=(Ra)_{d+e}$.  In the opposite
 direction, if $c\in(Ra)_{d+e}$ then we have $c=ma$ for some $m\in
 R_e$, and the rule $\mu_m(x)=mx$ defines an element
 $\mu_m\in D(Ra)_e$ with $\al(\mu_m)=c$.  This proves that the image
 of $\al$ is $\Sg^{-d}Ra$, as required.
\end{proof}

\begin{proposition}\label{prop-quotient}
 If $R$ is self-injective and $a\in R$ then $R/\ann(a)$ is also
 self-injective.
\end{proposition}
\begin{proof}
 Put $Q=R/\ann(a)$, and let $i\:Q\xra{}R$ be induced by $x\mapsto xa$,
 so $i$ is injective, with image $Ra$.  For $M\in\Mod_Q$ we write
 $D_Q(M)=\Hom_Q(M,Q)=\Hom_R(M,Q)$ and $D_R(M)=\Hom_R(M,R)$.  We are
 given that $D_R$ is exact, and we must show that $D_Q$ is exact.  The
 map $i\:Q\xra{}R$ gives a natural monomorphism $i\:D_Q(M)\to D_R(M)$,
 and it will suffice to show that this is also an epimorphism.  For
 any $\phi\:M\xra{}R$ we see that
 $\ann(a).\phi(M)=\phi(\ann(a)M)=\phi(0)=0$, so
 $\phi(M)\leq\ann_R^2(a)=Ra$, and $i\:Q\xra{}Ra$ is an isomorphism, so
 $\phi=i(\psi)$ for some $\psi\in D_Q(M)$, as required.
\end{proof}

\begin{proposition}\label{prop-ann-prop}
 If $R$ is self-injective and $I$ and $J$ are ideals in $R$ then
 $\ann_R(I+J)=\ann_R(I)\cap\ann_R(J)$ and
 $\ann_R(I\cap J)=\ann_R(I)+\ann_R(J)$.
\end{proposition}
\begin{proof}
 There is a short exact sequence
 \[ R/(I\cap J) \xra{\bsm 1\\1\esm} R/I\oplus R/J
     \xra{\bsm 1 & -1 \esm} R/(I+J).
 \]
 By applying the exact functor $D$, we get a short exact
 sequence 
 \[ \ann_R(I\cap J) \xla{\bsm 1 &1\esm} \ann_R(I)\oplus\ann_R(J)
     \xla{\bsm 1 \\ -1 \esm} \ann_R(I+J).
 \]
 The claim follows.
\end{proof}

\begin{corollary}\label{cor-ideal-intersection}
 If $R$ is local and self-injective and $I$ and $J$ are nontrivial
 ideals, then $I\cap J$ is also nontrivial.
\end{corollary}
\begin{proof}
 Let $\mxi$ be the maximal ideal.  As $I$ and $J$ are nontrivial we
 have $\ann(I)<R$ and $\ann(J)<R$, so $\ann(I)\leq\mxi$ and
 $\ann(J)\leq\mxi$, so $\ann(I\cap J)=\ann(I)+\ann(J)\leq\mxi<R$, so
 $I\cap J$ is nontrivial.
\end{proof}

\section{Criteria for self-injectivity}
\label{sec-criteria}

We first record a graded version of the standard Baer criterion for
injectivity. 
\begin{definition}
 Let $R$ be a graded ring, and let $I$ be a graded $R$-module.
 We say that $I$ satisfies the \emph{Baer condition} if  for every
 graded ideal $J\leq R$, every integer $d$ and every $R$-module
 homomorphism $\phi\:\Sg^dJ\to I$, there exists $m\in I_d$ such
 that $\phi(a)=am$ for all $a\in I$.  We say that $I$ satisfies
 the \emph{finite Baer condition} if the same condition holds for all
 finitely generated graded ideals $J$.
\end{definition}

\begin{proposition}\label{prop-baer-graded}
 In the above context, the module $I$ is injective if and only if it
 satisifes the Baer condition.
\end{proposition}
\begin{proof}
 This was originally done in the ungraded context in~\cite{ba:agd}, as
 an application of Zorn's Lemma.  The proof is also given in many
 textbooks such as~\cite{la:lmr}*{page 63}.  It can be modified in an
 obvious way to keep track of gradings, which gives our statement
 above.  
\end{proof}

\begin{proposition}\label{prop-finite-baer}
 Suppose that $I_d$ is finite for all $d$, and that $I$ satisfies
 the finite Baer condition.  Then $I$ also satisfies the full Baer
 condition and so is injective.
\end{proposition}
\begin{proof}
 Consider a graded ideal $J\leq R$ and a homomorphism
 $\phi\:\Sg^dJ\to I$.  For each finitely generated ideal
 $K\sse J$ we put 
 \[ M(K) = \{m\in I_d\st \phi(a)=am \text{ for all } a\in K\}. \]
 The finite Baer condition means that this is a nonempty subset of the
 finite set $I_d$.  Choose $K$ such that $|M(K)|$ is as small as
 possible, and choose $m\in M(K)$.  For $a\in J$ it is clear that
 $M(K+Ra)\sse M(K)$, so by the minimality property we must
 have $M(K+Ra)=M(K)$, so $m\in M(K+Ra)$, so $\phi(a)=am$.
 This proves the full Baer condition. 
\end{proof}

\begin{definition}\label{defn-block}
 Let $R$ be a graded ring, and let $I$ be an $R$-module.  A
 \emph{test pair} of length $r$ and degree $d$ is a pair $(u,v)$ where
 $u\in R^r$ and $v\in I^r$ such that the entries $u_i$ and $v_i$ are
 homogeneous with $|v_i|=|u_i|+d$ for all $i$.  A \emph{block} for
 such a pair is a vector $b\in R^r$ such that $b.u=0$ but $b.v\neq 0$
 (where $b.x=\sum_ib_ix_i$).  A \emph{transporter} is an element
 $m\in I_d$ such that $v_i=mu_i$ for all $i$.
\end{definition}
\begin{remark}\label{rem-separate-zeros}
 We implicitly formulate the theory of graded groups in such a way that
 the zero elements in different degrees are distinct.  Thus, the
 notation $|u|$ is meaningful even if $u=0$.
\end{remark}

\begin{proposition}\label{prop-block}
 The module $I$ satisfies the finite Baer condition iff every test
 pair has either a block or a transporter.
\end{proposition}
\begin{proof}
 Suppose that every test pair has either a block or a transporter.
 Consider a finitely generated graded ideal $J\leq R$, and a
 homomorphism $\phi\:\Sg^dJ\to R$.  Choose a list
 $u=(u_1,\dotsc,u_r)$ of homogeneous elements that generates $J$,
 and put $v_i=\phi(u_i)\in I$.  Note that if $b\in R^r$ with $b.u=0$
 then we can apply $\phi$ to see that $b.v=0$.  It follows that the
 pair $(u,v)$ has no block, so it must have a transporter.  This means
 that there is an element $m\in I_d$ with $\phi(u_i)=u_im$ for all
 $i$, and it follows easily that $\phi(a)=am$ for all $a\in J$, as
 required.  

 Conversely, suppose that $I$ satisfies the finite Baer condition.
 Consider a test pair $(u,v)$ of degree $d$ with no block, and let
 $J$ be the ideal generated by the entries $u_i$.  Define
 $\phi\:\Sg^dJ\to I$ by $\phi(\sum_ib_iu_i)=\sum_ib_iv_i$ (the
 absence of a block means that this is well-defined).  The finite Baer
 condition means that there is an element $m\in I_d$ with $\phi(a)=am$
 for all $a\in J$, and this $m$ is clearly a transporter for
 $(u,v)$. 
\end{proof}

\begin{corollary}\label{cor-limit-selfinj}
 Let $R$ be a graded commutative ring such that $R_k$ is finite for
 all $k$.  Suppose also that there are subrings
 \[ R(0) \leq R(1) \leq R(2) \leq \dotsb \leq R \]
 such that each $R(n)$ is self-injective and $R=\bigcup_nR(n)$.
 Then $R$ is self-injective. 
\end{corollary}
\begin{proof}
 Any test pair $(u,v)\in R^r\tm R^r$ can be regarded as a test
 pair over $R(n)$ for sufficiently large $n$.  As $R(n)$ is
 self-injective, there must be a block in $R(n)^r$ or a transporter
 in $R(n)$.  It is clear from the definitions that such a block or
 transporter still qualifies as a block or transporter over $R$, so
 we see that $R$ satisfies the finite Baer condition.  As we have
 assumed that $R_k$ is finite for all $k$, we can use
 Proposition~\ref{prop-finite-baer} to see that $R$ is injective as
 an $R$-module.
\end{proof}

\begin{theorem}\label{thm-selfinj-ann}
Let $R$ be a graded commutative ring such that $R_k$ is finite for
 all $k$. Then the following are equivalent:
  \begin{itemize}
  \item[(a)] $R$ is self-injective.
  \item[(b)] For all finitely generated ideals $J,K\leq R$ we have
   $\ann_R^2(J)=J$ and 
   \[\ann_R(J\cap K) = \ann_R(J)+\ann_R(K).
   \]
  \item[(c)] For all elements $a \in R$ and every finitely generated
   ideal $J \leq R$ we have $\ann_R^2(a)=Ra$ and
   \[\ann_R(J \cap Ra) =  \ann_R(J) + \ann_R(a).\]
 \end{itemize}
\end{theorem}
\begin{proof}
 It follows from Remark ~\ref{rem-ideal} and
 Proposition~\ref{prop-ann-prop} that~(a) implies~(b). If ~(b) holds,
 then ~(c) follows immediately. Now suppose ~(c) holds. As we have
 assumed that $R_k$ is finite for all $k$, we may use the theory of
 blocks and transporters. We proceed by induction on the length of a
 test pair to show that every test pair over the ring $R$ has either a
 block or a transporter. Let $(u;v)$ be a test pair of length $1$ and
 degree $d$. Suppose this test pair has neither block nor
 transporter. Then $\ann_R(u) \leq \ann_R(v)$ and by assumption we
 have $Rv = \ann^2_R(v) \leq \ann^2_R(u) = Ru$, that is, $v = um$ for
 some $m \in R_d$.  Since $m$ is a transporter for this test pair, we
 have a contradiction.

 Now suppose each test pair of length $ \leq k$ and arbitrary degree
 has either a block or a transporter. A test pair of length $k+1$ and
 degree $d$ takes the form $ (u,u_{k+1};v,v_{k+1})$ where $(u;v)$ is a
 test pair of length $k$ and degree $d$ and $(u_{k+1},v_{k+1})$ is a
 test pair of length $1$ and degree $d$.  By the inductive hypothesis,
 both the test pairs $(u;v)$ and $(u_{k+1},v_{k+1})$ have either a
 block or a transporter. If $(u;v)$ has block $r$, then $(r,0)$ is a
 block for the test pair $(u,u_{k+1};v,v_{k+1})$. Similarly, if
 $(u_{k+1},v_{k+1})$ has block $r_{k+1}$, then $(0,\ldots, 0,
 r_{k+1})$ is a block for the test pair
 $(u,u_{k+1};v,v_{k+1})$. Otherwise, $(u;v)$ must have transporter $m
 \in R_d$ and $(u_{k+1},v_{k+1})$ must have transporter $n \in R_d$.
 In this situation, suppose the test pair $(u,u_{k+1};v,v_{k+1})$ has
 neither block nor transporter and let $J$ be the ideal generated by
 the entries of $u$.  The absence of a block implies that there is a
 well defined map $\phi: \Sigma^d (J + Ru_{k+1}) \to R$ defined by
 $\phi(\sum^{k+1}_{i=1}b_iu_i)=\sum^{k+1}_{i=1}b_iv_i$.  Now let $s$
 be an element in the intersection $J \cap Ru_{k+1}$. Then we must
 have $s= \sum_{i=1}^k s_iu_i=s_{k+1}u_{k+1}$ for elements $s_i \in R$
 for each $i$. Applying the map $\phi$ to the zero element
 $(\sum_{i=1}^k s_iu_i) - s_{k+1}u_{k+1}$ gives
 \[
  0=\left(\sum_{i=1}^k s_iv_i \right)- s_{k+1}v_{k+1} =
    \left(\sum_{i=1}^k s_iu_im \right) - s_{k+1}u_{k+1}n 
   = s(m - n).
 \]
 Thus it follows that the element $m-n$ is in the annihilator ideal
 $\ann_R(J \cap Ru_{k+1})$. By assumption, we have $\ann_R(J \cap
 Ru_{k+1}) = \ann_R(J) + \ann_R(u_{k+1})$. Now let $m-n = x -y$ where
 $x \in \ann_R(J)$ and $y \in \ann_R(u_{k+1})$ and put $z = m-x = n -
 y$. Since $u_iz = u_i(m-x) = u_im = v_i$ for each $i \leq k$ and
 $u_{k+1}z = u_{k+1}(n-y)= u_{k+1}n = v_{k+1}$ it follows that $z$ is
 a transporter for the test pair $(u,u_{k+1};v,v_{k+1})$. As this
 gives a contradiction, it follows that every test pair of length
 $k+1$ and arbitrary degree must have either a block or
 transporter. We deduce that every test pair in the ring $R$ must have
 either a block or transporter, and since $R_k$ is finite for each
 $k$, we can use Proposition ~\ref{prop-block} to show that $R$ is
 injective as an $R$-module.
\end{proof}

\section{The Noetherian case}
\label{sec-noetherian}

\begin{theorem}\label{thm-noeth}
 Let $R$ be a Noetherian graded commutative ring.  Then the
 following are equivalent:
 \begin{itemize}
  \item[(a)] $R$ is self-injective.
  \item[(b)] For every ideal $J\leq R$ we have
   $\ann_R^2(J)=J$.
  \item[(c)] $R$ is Artinian (and thus is a finite product of
   Artinian local rings), and each of the local factors has
   one-dimensional socle.    
 \end{itemize}
\end{theorem}

Statements similar to this are certainly well-known (see for
example~\cite{brhe:cmr}*{Exercise 3.2.15}), but we do not know a
reference for this precise formulation.  For completeness we will give
a self-contained proof after some lemmas.

\begin{lemma}\label{lem-soc-min}
 Let $R$ be an Artinian local graded ring, with maximal ideal $\mxi$,
 and put $K=R/\mxi$.  Suppose that the socle $\soc(R)=\ann_R(\mxi)$
 has dimension one over $K$.  Then every nonzero ideal in $R$ contains
 $\soc(R)$.  
\end{lemma}
\begin{proof}
 Let $I$ be a nonzero ideal.  By the Artinian condition, we can choose
 an ideal $J$ that is minimal among nonzero ideals contained in $I$.
 Recall that every Artinian ring is Noetherian (see for
 example~\cite{ma:crt}*{Theorem 3.2}), so we can use Nakayama's Lemma
 to see that $\mxi J<J$ and thus (by minimality) that $\mxi J=0$.
 This means that $J$ is a nontrivial $K$-subspace of $\soc(R)$, but
 $\soc(R)$ has dimension one, so $J=\soc(R)$, so $\soc(R)\leq I$.
\end{proof}

\begin{lemma}\label{lem-soc-dac}
 Suppose that $R$ is as in Lemma~\ref{lem-soc-min}.  Then for all
 ideals $J\leq R$ we have $\ann_R^2(J)=J$. 
\end{lemma}
\begin{proof}
 First, it is standard that we can fit together a composition series
 for $J$ with a composition series for $R/J$ to get a chain
 \[ 0=I_0<I_1 < \dotsb < I_r = R \]
 with $I_i/I_{i-1}\simeq K$ for all $i$, and $J=I_t$ for some $t$.
 Now let $A_j$ be the annihilator of $I_j$, so we have 
 \[ R = A_0 \geq A_1 \geq \dotsb \geq A_r = 0. \]
 Now $\mxi A_iI_{i+1}=A_i(\mxi I_{i+1})\leq A_iI_i=0$, so
 $A_iI_{i+1}\leq\soc(R)$.  On the other hand, we have $A_iI_i=0$ and
 $A_{i+1}I_{i+1}=0$.  We therefore have a natural map 
 \[ \xi_i \: A_i/A_{i+1} \to \Hom_K(I_{i+1}/I_i,\soc(R)) \]
 given by $\xi_i(a+A_{i+1})(b+I_i)=ab$.  It is clear from the
 definitions that this is injective, and the codomain is isomorphic to
 $K$, so $A_i/A_{i+1}$ is either $0$ or $K$.  It is standard that any
 two composition series have the same length, so we must have
 $A_i/A_{i+1}\simeq K$ for all $i$, so $A_i$ has length $r-i$.  After
 applying the same logic to the composition series
 $\{A_{r-i}\}_{i=0}^r$ we see that the ideal $\ann(A_i)=\ann^2(I_i)$
 has length $i$.  We also know that $I_i\leq\ann^2(I_i)$ and that
 $I_i$ also has length $i$; it follows that $I_i=\ann^2(I_i)$, as
 required. 
\end{proof}

\begin{corollary}\label{cor-dac-selfinj}
 Suppose that $R$ is as in Lemma~\ref{lem-soc-dac}.  Then $R$ is
 self-injective. 
\end{corollary}
\begin{proof}
 Consider an ideal $I\leq R$ and an $R$-module map $f\:I\to R$.
 Choose a composition series $0=J_0<J_1<\dotsb<J_r=I$.  We have
 $J_i/J_{i-1}\simeq K$ so we can find $a_i\in J_i\sm J_{i-1}$ such
 that $J_i=J_{i-1}+Ra_i$ with $\mxi a_i\leq J_{i-1}$.
 
 We will construct elements $x_0,\dotsc,x_r\in R$ such that
 $f(a)=ax_i$ for all $a\in J_i$.  We start with $x_0=0$.   Now suppose
 we have found $x_{i-1}$.  Put $u_i=f(a_i)-x_{i-1}a_i$.  Using the
 fact that $\mxi a_i\leq I_{i-1}$ we find that $\mxi u_i=0$, so
 $u_i\in\soc(R)$.  Next, we have $a_i\not\in I_{i-1}=\ann^2(I_{i-1})$,
 so $\ann(I_{i-1})a_i\neq 0$.  As every nontrivial ideal contains the
 socle, we see that $u_i\in\ann(I_{i-1})a_i$, so we can write
 $u_i=y_ia_i$ for some $y_i$ with $y_iI_{i-1}=0$.  We now put
 $x_i=x_{i-1}+y_i$.  By construction we have $f(a)=ax_i$ for
 $a\in I_{i-1}$ or for $a=a_i$, and it follows that this equation
 holds for all $a\in I_i$ as required.  At the end of the induction we
 have an element $x_r$ which fulfils Baer's criterion.
\end{proof}

\begin{proof}[Proof of Theorem~\ref{thm-noeth}]
 It follows from Remark~\ref{rem-ideal} that~(a) implies~(b).  Now
 suppose that~(b) holds.  Consider a descending chain of ideals
 $I_0\geq I_1\geq I_2\geq \dotsb$ in $R$.  The ideals $\ann(I_k)$ then
 form an ascending chain, which must eventually stabilise because $R$
 is Noetherian.  We can thus take annihilators again to see that the
 original chain also stabilises.  This shows that $R$ is Artinian.  It
 follows in a standard way that there are only finitely many maximal
 ideals, and that $R$ is the product of its maximal localisations.  We
 thus have a splitting $R=\prod_{i=1}^nR_i$ say, where each factor
 $R_i$ an Artinian local ring.  It follows that the lattice of ideals
 in $R$ is the product of the corresponding lattices for the factors
 $R_i$, and thus that each $R_i$ satisfies condition~(b).  We can thus
 reduce to the case where $R$ is local, with maximal ideal $\mxi$ say.  
 Recall that the socle is
 $\soc(R)=\{a\in R\st a\mxi=0\}=\ann_R(\mxi)$, which is naturally a
 vector space over the field $K=R/\mxi$.  If $\soc(R)$ were zero we
 would have $\mxi=\ann^2(\mxi)=\ann(\soc(R))=\ann(0)=R$, which is a
 contradiction.  We can therefore choose a nonzero element
 $u\in\soc(R)$.  We find that $Ku=Ru$ is a nonzero ideal in $R$, so
 $\ann(Ku)$ is a proper ideal containing $\ann(\soc(R))=\mxi$, so
 $\ann(Ku)=\mxi$ by maximality.  We can now take annihilators again to
 see that $Ku=\ann(\mxi)=\soc(R)$, so $\soc(R)$ is one-dimensional.
 This proves~(c).  

 Finally, we will assume~(c) and prove~(a).  It is again easy to
 reduce to the case where $R$ is local, and the local case is covered
 by Corollary~\ref{cor-dac-selfinj}.
\end{proof}

\begin{definition}
 Let $K$ be a field.  A \emph{Poincar\'e duality algebra} over $K$ is
 a graded commutative $K$-algebra $R$ equipped with a $K$-linear map
 $\tht\:R_d\to K$ for some $d\geq 0$ such that
 \begin{itemize}
  \item For $i<0$ or $i>d$ we have $R_i=0$
  \item $R_0=K$.
  \item For $0\leq i\leq d$ we have $\dim_K(R_i)<\infty$, and the map
   $(a,b)\mapsto\tht(ab)$ defines a perfect pairing between $R_i$ and
   $R_{d-i}$. 
 \end{itemize}
\end{definition}

\begin{proposition}\label{prop-poincare}
 Every Poincar\'e duality algebra is self-injective.
\end{proposition}
\begin{proof}
 Let $R$ be a Poincar\'e duality algebra of top dimension $d$, and put
 $\mxi=\bigoplus_{i>0}R_i$.  It is clear that $R/\mxi=K$ and
 $\mxi^{d+1}=0$, and it follows that $\mxi$ is the unique maximal
 ideal.  As $R$ has finite total dimension over $K$ it is clearly
 Artinian.  The perfect pairing condition implies that $\soc(R)=R_d$
 and that this has dimension one.  It follows by
 Theorem~\ref{thm-noeth} that $R$ is self-injective.

 Alternatively, for any $R$-module $M$ we can define a natural map
 \[ \tau \: \Hom_R(M,R) \to \Hom_K(M_d,K) \]
 by $\tau(\phi)=\tht\circ\phi_d$.  Using the perfectness of the
 pairing we see that this is an isomorphism.  As $K$ is a field, the
 functor $M\mapsto\Hom_K(M_d,K)$ is exact, and it follows that the
 functor $M\mapsto\Hom_R(R,R)$ is also exact, or in other words that
 $R$ is injective as an $R$-module.
\end{proof}

\begin{example}\label{eg-exterior}
 Put 
 \[ E = \F[x_0,x_1,x_2,\dotsc]/(x_i^2\st i\geq 0), \]
 with $|x_i|=2^i$.  For any finite set $I\subset\N$ we put
 $x_I=\prod_{i\in I}x_i$, so $|x_I|=\sum_{i\in I}2^i$ and the elements
 $x_I$ form a basis for $E$ over $\F$.  It follows that
 $E_k\simeq\F$ for all $k\geq 0$, and $E_k=0$ for $k<0$.  Let $E(n)$
 be the subalgebra of $E$ generated by $x_0,\dotsc,x_{n-1}$.  This
 is a Poincar\'e duality algebra, with socle generated by the element
 $\prod_{i<n}x_i$, and it is clear that $E=\bigcup_nE(n)$.
 Corollary~\ref{cor-limit-selfinj} therefore tells us that $E$ is
 self-injective. 
\end{example}

\section{Coherence}
\label{sec-coherence}

We now briefly recall some standard ideas about finite presentation.
\begin{definition}\label{defn-fp}
 Let $R$ be a graded commutative ring, and let $M$ be a graded
 $R$-module.  Then we see from~\cite{la:lmr}*{Section 4D} the
 following are equivalent: 
 \begin{itemize}
  \item[(a)] There exists an exact sequence 
   \[ \xymatrix{
     P_1 \ar[r]^f & P_0 \ar@{->}[r]^g & M \ar[r] & 0,
   } \]
   where $P_0$ and $P_1$ are finitely generated free modules.  
  \item[(b)] $M$ is finitely generated, and for every epimorphism
   $g\:P_0\to M$ (with $P_0$ a finitely generated free module) the
   module $\ker(g)$ is also finitely generated.
 \end{itemize}
 If these conditions hold, we say that $M$ is \emph{finitely
  presented}. 
\end{definition}
\begin{remark}
 By a \emph{finitely generated free module} we mean one of the form
 $\bigoplus_{i=1}^r\Sg^{d_i}R$; we do not assume that the degree shift
 $d_i$ is zero.
\end{remark}


\begin{corollary}
 If $R$ is Noetherian, then every finitely generated ideal is finitely
 presented. 
\end{corollary}
\begin{proof}
 Condition~(b) is clearly satisfied.
\end{proof}

As we stated in Definition~\ref{defn-coherent}, a graded ring $R$ is
said to be \emph{coherent} if every finitely generated ideal is
finitely presented, and \emph{totally incoherent} if the only finitely
presented ideals are $0$ and $R$.  It is clear that every Noetherian
ring is coherent.  We mention as background that if $R$ is coherent,
then the category of finitely generated modules is closed under
images, kernels, cokernels and extensions, so it is an abelien
category.  The following example is standard:

\begin{proposition}\label{prop-exterior-coherent}
 The infinite exterior algebra $E$ (as in Example~\ref{eg-exterior})
 is coherent. 
\end{proposition}
\begin{proof}
 Let $E(n)$ be the subalgebra generated by $x_0,\dotsc,x_{n-1}$, and
 let $E'(n)$ be generated by the remaining variables, so
 $E=E(n)\ot_{\F}E'(n)$.  Any finitely generated ideal is the image of
 some $E$-linear map $g\:E^r\to E$, which will have the form
 $g(u)=u.v$ for some vector $v\in E^r$.  We must show that the module
 $K=\ker(g)$ is finitely generated.  Choose $n$ large enough that
 $v_i\in E(n)$ for all $i$.  Now $v$ gives a map $g'\:E(n)^r\to E(n)$
 of $E(n)$-modules, and $E(n)$ is Noetherian, so the module
 $K'=\ker(g')$ is finitely generated over $E(n)$.  We can identify $g$
 with $g'\ot 1$ with respect to the splitting $E=E(n)\ot E'(n)$, and
 it follows that $K=K'\ot E(n)'$, and thus that any finite generating
 set for $K'$ over $E(n)$ also generates $K$ over $E$.
\end{proof}

The following result will be our main tool for proving incoherence
results. 

\begin{lemma}\label{lem-ann-fg}
 Let $A$ be a local graded ring, with maximal ideal $\mxi$, and let
 $I$ be a finitely presented ideal in $A$.  Then for each
 $u\in I\sm \mxi I$, the image of $\ann_A(u)$ in $\mxi/\mxi^2$ has
 finite dimension over $A/\mxi$.  
\end{lemma}

Note here that as $u\not\in\mxi I$ we have $u\neq 0$, so
$\ann_A(u)\leq\mxi$ and it is meaningful to talk about the image in
$\mxi/\mxi^2$.  

\begin{proof}
 As $I$ is finitely generated over $A$, we see that $I/\mxi I$ is a
 finite-dimensional vector space over $A/\mxi$.  We can choose a basis
 for this space containing the image of $u$, and then choose elements
 of $I$ lifting these basis elements.  This gives a list
 $v_1,\dotsc,v_n\in I$ with $v_1=u$ such that the corresponding map
 $g\:A^n\to I$ induces an isomorphism
 $\ov{g}\:(A/\mxi)^n\to I/\mxi I$.  Now $\cok(g)$ is a finitely
 generated module with $\mxi.\cok(g)=\cok(g)$, so $\cok(g)=0$ by
 Nakayama's Lemma, so $g$ is an epimorphism.  As $I$ is assumed to be
 finitely presented, we see that $\ker(g)$ is also finitely generated
 over $A$.  Moreover, as $\ov{g}$ is an isomorphism we see that
 $\ker(g)\leq\mxi^n$.  It follows that the image of $\ker(g)$ in
 $(\mxi/\mxi^2)^n$ is finite-dimensional.  The intersection of
 $\ker(g)$ with the first copy of $A$ in $A^n$ is just the annihilator
 of $u$, so we see that the image of $\ann_A(u)$ in $\mxi/\mxi^2$ is
 finite-dimensional.  
\end{proof}

\begin{corollary}\label{cor-incoherent}
 Let $A$ be a local graded ring, with maximal ideal $\mxi$.  Suppose
 that for all $u\in A$ we have either 
 \begin{itemize}
  \item[(a)] $u=0$; or
  \item[(b)] the image of $\ann_A(u)$ in $\mxi/\mxi^2$ has infinite
   dimension; or
  \item[(c)] $u$ is invertible.
 \end{itemize}
 Then $A$ is totally incoherent.
\end{corollary}
\begin{proof}
 Let $I$ be a finitely presented ideal.  If $\mxi I=I$ then $I=0$ by
 Nakayama's Lemma.  Otherwise, we can choose $u\in I\sm \mxi I$.  As
 $u\not\in\mxi I$ we have $u\neq 0$.  By the lemma, the image of
 $\ann_A(u)$ in $\mxi/\mxi^2$ must have finite dimension.  Thus,
 possibilities~(a) and~(b) are excluded, so $u$ must be invertible.
 As $u\in I$ we conclude that $I=A$.
\end{proof}
Next we record a graded version of Chase's Theorem for
coherent rings. 

\begin{theorem}\label{thm-chase-graded}
Let $R$ be a graded commutative ring. Then the following are equivalent:
  \begin{itemize}
  \item[(a)] $R$ is coherent.
  \item[(b)] For all elements $a\in R$ and for every finitely generated ideal $J \leq R$, the conductor ideal
  \[
  (J:a) = \{r \in R \st ra \in J\}
  \] is finitely generated.
  \item[(c)] For all elements  $a \in R$, the annihilator ideal $\ann_R(a)$ is finitely generated, and for all finitely generated ideals $J,K \leq R$,  
  the intersection $J \cap K$ is finitely generated.
 \end{itemize}
 \end{theorem}
\begin{proof}
 The ungraded version of the proof is given in many
 textbooks such as~\cite{la:lmr}*{page 142}.  It can be modified in an
 obvious way to keep track of gradings, which gives our statement
 above.  
\end{proof}

\begin{theorem}
Let $R$ be a graded commutative ring such that $R_k$ is finite for
 all $k$. Then the following are equivalent:
  \begin{itemize}
  \item[(a)] $R$ is coherent and self-injective.
  \item[(b)] $R$ is coherent and for all finitely generated ideals $J\leq R$ we have
   $\ann_R^2(J)=J$.
  \item[(c)] For every finitely generated ideal $J \leq R$, the ideal $\ann_R(J)$ is finitely generated and  
  $\ann^2_R(J)=J$.
  \item[(d)] $R$ is self injective and for all finitely generated ideals $J \leq R$, the ideal $\ann_R(J)$ is finitely generated.
 \end{itemize}
\end{theorem}
\begin{proof}

 It follows from Remark ~\ref{rem-ideal} that ~(a) implies ~(b). To
 show that ~(b) implies ~(c) we need to show that the ideal
 $\ann_R(J)$ is finitely generated for each finitely generated ideal
 $J \leq R$. If we let $(r_1, \ldots, r_n)$ be generators for the
 ideal $J$, then we can take the annihilator of $J$ to give
 $\ann_R(J) = \bigcap_i \ann_R(r_i)$. Since $R$ is assumed to be
 coherent, it follows from part ~(c) of Theorem \ref{thm-chase-graded}
 that $\ann_R(r_i)$ is finitely generated for each $i$ and that a
 finite intersection of finitely generated ideals is also finitely
 generated. Thus $\ann_R(J)$ is finitely generated as claimed. Now
 suppose that part ~(c) holds. To prove that ~(c) implies ~(d), we
 need to show that $R$ is injective as an $R$-module.  For all ideals
 $J,K \leq R$ we have
 \[
  \ann_R(\ann_R(J) + \ann_R(K)) = \ann^2_R(J) \cap \ann^2_R(K) = J \cap K.
 \]
 By assumption, the ideal sum $\ann_R(J) + \ann_R(K)$ must be finitely
 generated. Thus we can take double annihilators to give 
 \[ \ann_R(J) + \ann_R(K) = \ann_R(J \cap K). \]
 Since $R_k$ is finite for each $k$, we can use part ~(b) of Theorem
 \ref{thm-selfinj-ann} to complete the claim. We now conclude by
 showing that ~(d) implies ~(a). By assumption, the annihilator ideal
 $\ann_R(a)$ is finitely generated for all elements $a \in R$. Then
 for all ideals $J, K \leq R$ we know that the ideal sum $\ann_R(J) +
 \ann_R(K)$ is finitely generated by assumption. By taking
 annihilators we then have
 \[ \ann_R(\ann_R(J) + \ann_R(K)) =
     \ann^2_R(J) \cap \ann^2_R(K) = J \cap K
 \]
 where the double annihilator condition holds by
 Remark~\ref{rem-ideal}. However, by assumption, the annihilator of a
 finitely generated ideal is also finitely generated. Thus the
 intersection $J \cap K$ must be finitely generated.  It follows from
 part ~(c) of Theorem \ref{thm-chase-graded} that the ring $R$ is
 coherent as claimed. 
\end{proof}

\section{Self-injective adjustment}
\label{sec-adjust}

\begin{definition}
 We write $\CR$ for the category of commutative graded
 $\F$-algebras such that 
 \begin{itemize}
  \item[(a)] $R_k=0$ for all $k<0$.
  \item[(b)] $R_0=\F$.
  \item[(c)] $R_k$ is finite for all $k>0$.
 \end{itemize}
\end{definition}

\begin{proposition}\label{prop-adjoin-block}
 Let $R$ be a ring in $\CR$, and let $\CP$ be a finite set of test
 pairs in $R$ that have no transporters.  Let $m$ be a positive
 integer.  Then there is an extension $R'\geq R$ of graded
 rings such that 
 \begin{itemize}
  \item[(a)] $R'$ is also in $\CR$.
  \item[(b)] $R'_k=R_k$ for all $k<m$.
  \item[(c)] Each test pair in $\CP$ has a block in $R'$.
 \end{itemize}
\end{proposition}
\begin{proof}
 List the elements of $\CP$ as $(u_0,v_0),\dotsc,(u_{p-1},v_{p-1})$
 say.  Suppose that $(u_t,v_t)$ has length $r_t$, and let $d_t$ be the
 maximum of the degrees of the entries $u_{t,j}$ for $0\leq j<r_t$.
 Let $P$ be the polynomial ring obtained from $R$ by adjoining
 variables $b_{t,j}$ for $0\leq t<p$ and $0\leq j<r_t$, with
 $|b_{t,j}|=m+d_t-|u_{t,j}|\geq m>0$.  Put
 $w_t=\sum_{j=0}^{r_t-1}b_{t,j}u_{t,j}\in P$ and 
 $R'=P/(w_0,\dotsc,w_{p-1})$.  There is an evident ring map
 $\eta\:R\to R'$, and also a ring map $\pi\:R'\to R$ given by
 $\pi(b_{t,j})=0$ for all $t$ and $j$.  It is clear that $\pi\eta=1$,
 so $\eta$ is injective, and we can use it to regard $R'$ as an
 extension of $R$.  As $|b_{t,j}|\geq m>0$, it is easy to see that
 $R'\in\CR$ and that the map $R_k\to R'_k$ is surjective (and
 therefore bijective) for $k<m$.  By construction we have $b_t.u_t=0$
 in $R'$.  We claim that $b_t.v_t\neq 0$ in $R'$, or equivalently
 that $b_t.v_t$ cannot be written as $\sum_s c_sw_s$ in $P$.  To see
 this, let $c^*$ denote the constant term in the polynomial $c_t$.
 By examining the coefficient of $b_{t,j}$ in the equation
 $b_t.v_t=\sum_sc_sw_s$ we obtain $v_{t,j}=c^*u_{t,j}$ for all $j$,
 which means that $c^*$ is a transporter for $(u_t,v_t)$, contrary to
 assumption.  Thus, $b_t$ is a block for $(u_t,v_t)$ in $R'$, as
 required. 
\end{proof}

\begin{definition}\label{defn-good}
 Let $R$ be a ring in $\CR$, and let $(u,v)$ be a test pair for
 $R$.  We say that $(u,v)$ is \emph{good} if it has either a block
 or a transporter, and \emph{bad} otherwise.  We say that $(u,v)$ is
 \emph{nondegenerate} if $u_i\neq 0$ for all $i$.  For any homogeneous
 element $x\in R$ we put $|x|_+=\max(0,|x|)$.  The \emph{weight}
 of $(u,v)$ is $\sum_i(1+|u_i|_++|v_i|_+)$.
\end{definition}

\begin{lemma}\label{lem-nondegenerate}
 Let $R$ be a ring in $\CR$, and suppose that all nondegenerate
 test pairs are good.  Then $R$ is self-injective. 
\end{lemma}
\begin{proof}
 Consider an arbitrary test pair $(u,v)\in R^r\tm R^r$.  If there
 exists $i$ such that $u_i=0$ but $v_i\neq 0$, then the basis vector
 $e_i\in R^r$ is a block for $(u,v)$.  Otherwise, let $(u',v')$ be
 the test pair obtained by removing all zeros from $u$ and the
 corresponding zeros from $v$.  This is nondegenerate, so it has a
 block or a transporter.  If $b'$ is a block for $(u',v')$, then we
 can construct a block for $(u,v)$ by inserting some zeros.  If $m'$
 is a transporter for $(u',v')$, then it is also a transporter for
 $(u,v)$.  We therefore see that all test pairs for $R$ are good, so
 $R$ is self-injective.
\end{proof}

\begin{lemma}\label{lem-count-bad}
 There are only finitely many nondegenerate bad test pairs of any
 given weight. 
\end{lemma}
\begin{proof}
 Consider an integer $N\geq 0$.  Any nondegenerate bad test pair
 $(u,v)$ of weight $N$ must have length at most $N$.  Moreover, as
 $(u,v)$ is nondegenerate we must have $u_i\neq 0$ for all $i$, and as
 $R\in\CR$ this means that $|u_i|\geq 0$.  We also have
 $\sum_i|u_i|\leq\wgt(u,v)=N$.  It is clear from this (and the
 finiteness of $R_k$) that there are only finitely many possibilities
 for $u$.  Next, let $d$ be the degree of $(u,v)$, so
 $|v_i|=|u_i|+d$.  From this it is clear that $d\leq N$.  If $d$ is
 sufficiently negative then we will have $v_i=0$ for all $i$, so $0$
 is a transporter for $(u,v)$, contradicting the assumption that
 $(u,v)$ is bad.  We therefore see that there are only finitely many
 possibilities for $d$.  Given $u$ and $d$, it is clear that there are
 only finitely many possibilities for $v$.
\end{proof}

\begin{theorem}\label{thm-adjust}
 Suppose that $R\in\CR$, and that $m\geq 0$.  Then there is an
 extension $R'\geq R$ such that 
 \begin{itemize}
  \item[(a)] $R'$ is also in $R$.
  \item[(b)] $R'_k=R_k$ for all $k<m$.
  \item[(c)] $R'$ is self-injective.
 \end{itemize}
\end{theorem}
\begin{proof}
 We define rings $R'(0)\leq R'(1)\leq \dotsb$ as follows.  We start
 with $R'(0)=R$.  For each $k\geq 0$, we let $R'(k+1)$ be an
 extension of $R'(k)$ that agrees with $R'(k)$ in degrees less than
 $k+m$, such that every nondegenerate bad test pair of weight at most
 $k$ in $R'(k)$ has a block in $R'(k+1)$.  This can be constructed by
 Proposition~\ref{prop-adjoin-block} and Lemma~\ref{lem-count-bad}.
 Now take $R'$ to be the colimit of the rings $R'(k)$.  By
 construction we have $R'_i=R'(k)_i$ for sufficiently large $k$, and
 using this it is clear that $R'\in\CR$.  Consider a nondegenerate
 test pair $(u,v)\in R'$.  For sufficiently large $k$ we can assume
 that $k\geq\wgt(u,v)$ and that $u_i,v_i\in R'(k)$ for all $i$.
 If $(u,v)$ is good in $R'(k)$ then it is good in $R'$.  If it is
 bad in $R'(k)$ then by construction it becomes good in $R'(k+1)$
 and therefore in $R'$.   
\end{proof}

\section{The cube algebra}
\label{sec-cube}

Recall that in the statement of Theorem~\ref{thm-cube-intro} we
introduced the ring
\[ C = \F[y_0,y_1,\dotsc]/(y_i^3+y_iy_{i+1}\st i\geq 0), \]
with the grading given by $|y_i|=2^i$.  We now investigate the
structure of this ring (which we call the \emph{cube algebra}).

\begin{definition}\label{defn-Cn}
 We also put
 \begin{align*}
  C[n,\infty] &= \F[y_n,y_{n+1},\dotsc]/
   (y_i^3+y_iy_{i+1}| n\leq i<\infty) \\
  C[n,m] &= \F[y_n,\dotsc,y_m]/
   (y_i^3+y_iy_{i+1}| n\leq i< m) \\
  \bC[n,m] &= C[n,m]/y_m.
 \end{align*}
\end{definition}

\begin{lemma}\label{lem-Cn-subring}
 The evident maps
 \[ \xymatrix{
  C[n+1,m] \ar[r] \ar[d] &
  C[n+1,m+1] \ar[r] \ar[d] &
  C[n+1,\infty] \ar[d] \\
  C[n,m] \ar[r] \ar[d] &
  C[n,m+1] \ar[r] \ar[d] &
  C[n,\infty] \ar[d] \\
  C[0,m] \ar[r] &
  C[0,m+1] \ar[r] &
  C[0,\infty] = C
 } \]
 are all split injective, so all the rings mentioned can be considered
 as subrings of $C$.
\end{lemma}
\begin{proof}
 There is a graded ring map $\tau_0\:\F[y_0,y_1,\dotsc]\to C[n,m]$
 given by  
 \[ \tau_0(y_i) = \begin{cases}
                 0 & \text{ if } i < n \\
                 y_i & \text{ if } n \leq i \leq m \\
                 y_m^{2^{i-m}} & \text{ if } m \leq i.
                \end{cases}
 \]
 It is straightforward to check that $\tau_0(y_i^3+y_iy_{i+1})=0$ for
 all $i\geq 0$, so there is an induced map $\tau\:C\to C[n,m]$.
 It is clear that the composite $C[n,m]\to C\xra{\tau} C[n,m]$ is the
 identity, so the map $C[n,m]\to C$ is injective for all $m$ and
 $n$.  The other claims follow from this.
\end{proof}

\begin{definition}\label{defn-multiindex}
 We write $P$ for the polynomial ring $\F[y_0,y_1,\dotsc]$, so
 that $C$ is a quotient of $P$.  A \emph{multiindex} is a sequence
 $\al=(\al_0,\al_1,\dotsc)$ of natural numbers with $\al_i=0$ for
 $i\gg 0$.  We write $MP$ for the set of all multiindices.  Given
 $\al\in MP$ we write $y^\al=\prod_iy_i^{\al_i}$ and
 $|\al|=|y^\al|=\sum_i\al_i2^i$.  It is clear that the set
 $BP=\{y^\al\st\al\in MP\}$ is a basis for $P$ over $\F$. 
\end{definition}

\begin{definition}
 We put 
 \begin{align*}
  M'C[n,m] &= \{\al\in MP\st \al_i=0 \text{ for } i<n \text{ or }
                i>m \text{ and } \al_i<3 \text{ for } n\leq i<m\} \\
  M\bC[n,m] &= \{\al\in MP\st \al_i=0 \text{ for } i<n \text{ or }
                i\geq m\} \\
  B'C[n,m] &= \{y^\al\st \al\in M'C[n,m]\} \\
  B\bC[n,m] &= \{y^\al\st \al\in M\bC[n,m]\}.
 \end{align*}
\end{definition}

Note that in the definition of $M'C[n,m]$ the constraint $\al_i<3$
does not apply when $i=m$, so in particular $M'C[n,m]$ is infinite.

\begin{proposition}\label{prop-cube-basis-i}
 $B'C[n,m]$ is a basis for $C[n,m]$, and $B\bC[n,m]$ is a basis for
 $\bC[n,m]$.  Moreover, $\bC[n,m]$ is a Poincar\'e duality algebra
 over $\F$. 
\end{proposition}

The proof depends on the following result:

\begin{lemma}\label{lem-monic-extension}
 Let $A$ be a commutative algebra over $\F$, let $f(t)\in A[t]$ be a
 monic polynomial of degree $d$, and put $B=A[x]/f(x)$.  Then
 $\{1,x,\dotsc,x^{d-1}\}$ is a basis for $B$ over $A$.  Moreover, if
 $A$ is finite-dimensional over $\F$ and has Poincar\'e duality, then
 the same is true of $B$.
\end{lemma}
\begin{proof}
 We first claim that any polynomial $g(x)\in A[x]$ can be expressed
 uniquely in the form $g(x)=q(x)f(x)+r(x)$ with $\deg(r(x))<d$.  This
 can easily be proved by induction on the degree of $g(x)$, and it
 follows directly that $\{1,\dotsc,x^{d-1}\}$ is a basis for $B$ over
 $A$.  Now suppose that $A$ has Poincar\'e duality, so there is a
 linear map $\tht\:A\to\F$ such that the bilinear form
 $(u,v)\mapsto\tht(u,v)$ is perfect.  This means that there exist
 bases $\{u_0,\dotsc,u_{n-1}\}$ and $\{v_0,\dotsc,v_{n-1}\}$ for $A$
 such that $\tht(u_iv_j)=\dl_{ij}$.  Now define $\phi\:B\to\F$ by
 $\phi(\sum_{i=0}^{d-1}a_ix^i)=\tht(a_{d-1})$.  We define bases
 $\{s_0,\dotsc,s_{nd-1}\}$ and $\{t_0,\dotsc,t_{nd-1}\}$ for $B$ by
 $s_{ni+j}=x^iu_j$ and $t_{ni+j}=x^{d-1-i}v_j$ for $0\leq i<d$ and
 $0\leq j<n$.  It is clear that $\phi(s_kt_k)=1$.  Suppose we have
 $0\leq k<k'<nd$.  Write $k=ni+j$ and $k'=ni'+j'$ as before; we must
 have either $i<i'$, or ($i=i'$ and $j<j'$).  In either case, we find
 that $\phi(s_it_j)=0$.  Thus, the matrix of $\phi$ with respect to
 our bases is triangular, with ones on the diagonal, proving that
 $\phi$ gives a perfect pairing on $B$.
\end{proof}

\begin{proof}[Proof of Proposition~\ref{prop-cube-basis-i}]
 From the definitions we have $C[m,m]=\F[y_n]$ and
 $B'C[m,m]=\{y_n^{\al_n}\st\al_n\in\N\}$ so it is clear that
 $B'C[m,m]$ is a basis for $C[m,m]$.  Similarly, it is clear that the
 set $\bC[m,m]=\{1\}$ is a basis for the ring
 $\bC[m,m]=C[m,m]/y_m=\F$, and that this has Poincar\'e duality.  

 Next, $C[n,m]$ can be described as $C[n+1,m][y_n]/f(y_n)$, where
 $f(t)=t^3+y_{n+1}t$ is a monic polynomial of degree three with
 coefficients in $C[n+1,m]$.  It also follows that
 $\bC[n,m]=\bC[n+1,m][y_n]/f(y_n)$.  All claims in the proposition now
 follow by downwards induction on $n$ using
 Lemma~\ref{lem-monic-extension}.
\end{proof}

\begin{remark}\label{rem-duality}
 Note that the algebra
 \[ \bC[n,m] =
     \frac{\F[y_n,y_{n+1},\dotsc,y_{m-1}]}{
          (y_n^3+y_ny_{n+1},\dotsc,y_{m-1}^3)}
 \] 
 has the same number of relations as generators, and has finite
 dimension over $\F$.  It is known that in this situation the sequence
 of relations is necessarily regular, and that the algebra
 automatically has Poincar\'e duality.  (This can be extracted
 from~\cite{ma:crt}*{Section 17}, for example.)  This would give
 another approach to Proposition~\ref{prop-cube-basis-i}.
\end{remark}

\begin{definition}
 Let $\al$ be a multiindex.  We say that
 \begin{itemize}
  \item[(a)] $\al$ is \emph{flat} if $\al_i<3$ for all $i$;
  \item[(b)] $\al$ is \emph{$n$-truncated} if $\al_i=0$ for all $i<n$;
  \item[(c)] $\al$ is \emph{$m$-solid} if it is flat and whenever
   $m\leq p\leq q$ and $\al_q>0$ we also have $\al_p>0$.
 \end{itemize}
 We consider all flat multiindices to be $\infty$-solid.  For
 $0\leq n\leq m\leq\infty$ we put  
 \[ MC[n,m] = \{\al\in MP\st \al
        \text{ is $n$-truncated and  $m$-solid } \},
 \]
 and $BC[n,m]=\{y^\al\st\al\in MC[n,m]\}$.  We also write $MC$ for the
 set $MC[0,\infty]$ of all flat multiindices.
\end{definition}

\begin{proposition}\label{prop-cube-basis-ii}
 $BC[n,\infty]$ is a basis for $C[n,\infty]$.
\end{proposition}
\begin{proof}
 We must show that for each degree $d\in\N$, the set $BC[n,\infty]_d$
 is a basis for $C[n,\infty]_d$.  Choose $m>n$ such that $2^m>d$.  It
 is then clear that $BC[n,\infty]_d=B'C[n,m]_d$ and
 $C[n,\infty]_d=C[n,m]_d$ so the claim follows from
 Proposition~\ref{prop-cube-basis-i}. 
\end{proof}

It is also true that $BC[n,m]$ is a basis for $C[n,m]$ when
$m<\infty$, but it is convenient to leave the proof until later.

\begin{proposition}\label{prop-flatten}
 For any multiindex $\al\in MP$, there is a multiindex $\bt\in MC$
 such that $y^\al=y^\bt$.
\end{proposition}
\begin{proof}
 If $\al\not\in MC$, we let $k$ denote the smallest index such that
 $\al_k>2$, and define $\al'\in MP$ by 
 \[ \al'_i = 
     \begin{cases}
      \al_i       & \text{ if } i < k \\
      \al_k-2     & \text{ if } i = k \\
      \al_{k+1}+1 & \text{ if } i = k+1 \\
      \al_i       & \text{ if } i > k+1.
     \end{cases}
 \]
 Because $y_k^3=y_ky_{k+1}$ we have $y^\al=y^{\al'}$.  Moreover,
 $\al'$ has the same degree as $\al$, and is lexicographically lower
 than $\al$.  There are only finitely many monomials of any given
 degree, so the claim follows by induction over the lexicographic
 order. 
\end{proof}

\begin{definition}\label{defn-x}\ \\
 \begin{itemize}
  \item[(a)] We put $x_0=y_0$, and $x_n=y_n+y_{n-1}^2$ for all $n>0$.
  \item[(b)] For $n\geq m\geq 0$ we put $x_{[m,n]}=\prod_{i=n}^mx_i$
   and $y_{[m,n]}=\prod_{i=n}^my_i$.
 \end{itemize}
\end{definition}

\begin{proposition}\label{prop-cube-x}
 For all $n\geq 0$ we have $y_n=\sum_{i=0}^nx_{n-i}^{2^i}$ and
 $y_nx_{n+1}=0$.  Thus, the ring $C$ can also be presented as 
 \[ C = \F[x_0,x_1,x_2,\dotsc]/
        (x_{n+1}\sum_{i=0}^nx_{n-i}^{2^i}\st n\geq 0).
 \]
\end{proposition}
\begin{proof}
 Once we recall that $(a+b)^2=a^2+b^2\pmod{2}$, the equation
 $y_n=\sum_{i=0}^nx_{n-i}^{2^i}$ is easily checked by induction.  Note
 that this already holds in the polynomial ring $P$.  As the elements
 $x_i$ can be expressed in terms of the $y_j$ and vice-versa, we see
 that $P=\F[x_0,x_1,\dotsc]$.  The defining relations
 $y_n^3+y_ny_{n+1}=0$ for $C$ can clearly be rewritten as
 $y_nx_{n+1}=0$ and thus as $x_{n+1}\sum_{i=0}^nx_{n-i}^{2^i}=0$.
\end{proof}

\begin{lemma}\label{lem-flatten}
 Whenever $m\leq n$ we have $y_my_{[m,n]}^2=y_{[m,n+1]}$.
\end{lemma}
\begin{proof}
 The inductive step is
 \[ y_my_{[m,n+1]}^2 = 
    y_my_{[m,n]}^2 y_{n+1}^2 = 
    y_{[m,n+1]} y_{n+1}^2 = 
    y_{[m,n]} y_{n+1}^3 = 
    y_{[m,n]}y_{n+1}y_{n+2} = 
    y_{[m,n+2]}.
 \]
\end{proof}

\begin{corollary}\label{cor-y-powers-i}
 For $k\geq 0$ we have $y_m^{2^k-1}=y_{[m,m+k-1]}$.
\end{corollary}
\begin{proof}
 The induction step is 
 \[ y_m^{2^{k+1}-1} = 
    y_m \left(y_m^{2^k-1}\right)^2 = 
    y_m y_{[m,m+k-1]}^2 = 
    y_{[m,m+k]}.
 \]
\end{proof}

\begin{lemma}\label{lem-y-powers}
 Fix $m\in\N$, and put 
 \[ U =
    \{\al\in MC\st \al \text{ is $m$-solid and $\al_i=0$ for $i<m$}\}.
 \]
 Then there is a bijection $\N\to U$ written as $k\mapsto\tht[m,k]$
 such that $y^{\tht[m,k]}=y_m^k$ in $C$.
\end{lemma}
\begin{proof}
 First, if $\al\in U$ it is clear that $|\al|$ is divisible by $2^m$,
 so we can define $\dl\:U\to\N$ by $\dl(\al)=|\al|/2^m$.  

 Now consider $k\in\N$.  There is a unique $r\in\N$ such that
 $2^r-1\leq k<2^{r+1}-1$.  This means that $0\leq k-(2^r-1)<2^r$, so
 there is a unique set $J\sse\{0,1,\dotsc,r-1\}$ with
 $k-(2^r-1)=\sum_{j\in J}2^j$.  We put
 \[ \tht[m,k]_i =
     \begin{cases}
      0 & \text{ if } i < m \\
      1 & \text{ if } m\leq i <m+r \text{ and } i-m\not\in J \\
      2 & \text{ if } m\leq i <m+r \text{ and } i-m\in J \\
      0 & \text{ if } m+r\leq i.
     \end{cases}
 \]
 This is clearly in $U$.  Next, we claim that $y^{\tht[m,k]}=y_m^k$.
 To see this, put $z=y_m^{2^r-1}$, which is the same
 as $y_{[m,m+r-1]}$ by Corollary~\ref{cor-y-powers-i}.  We have 
 \begin{align*}
  y^{\tht[m,k]} &= y_{[m,m+r-1]} \prod_{j\in J} y_{m+j}
                 = z \prod_{j\in J} y_{m+j} \\
  y_m^k &= y_m^{2^r-1+\sum_{j\in J}2^j}
         = z\prod_{j\in J}y_m^{2^j}
 \end{align*}
 Now, for $0\leq j<r$ we have $y_{m+j}(y_{m+j}^2+y_{m+j+1})=0$ and $z$
 is divisible by $y_{m+j}$ so $z(y_{m+j}^2+y_{m+j+1})=0$, so
 $y_{m+j+1}=y_{m+j}^2$ modulo $\ann(z)$.  It follows inductively that
 $y_{m+j}=y_m^{2^j}\pmod{\ann(z)}$, so
 $\prod_{j\in J}y_{m+j}=\prod_{j\in J}y_m^{2^j}\pmod{\ann(z)}$, so
 $y^{\tht[m,k]}=y_m^k$ as claimed.  It also follows that
 $\dl(\tht[m,k])=|y^{\tht[m,k]}|/2^m=|y_m^k|/2^m=k$.

 Now let $\al$ be an arbitrary element of $U$.  By the definition of
 solidity, there is an integer $s\geq 0$ such that when $m\leq i<m+s$
 we have $\al_i\in\{1,2\}$ and for $i\geq m+s$ we have $\al_i=0$.  It
 is then clear that 
 \[ \sum_{m\leq i<m+s}2^i\leq|\al|\leq 2\sum_{m\leq i<m+s}2^i, \]
 or in other words $2^s-1\leq\dl(\al)<2^{s+1}-1$.  It follows easily
 that $\al=\tht[m,\dl(\al)]$, so we have a bijection as claimed.
\end{proof}

\begin{proposition}\label{prop-cube-basis-iii}
 For $0\leq n\leq m\leq\infty$, the set $BC[n,m]$ is a basis for
 $C[n,m]$.  
\end{proposition}
\begin{proof}
 The case $m=\infty$ was covered by Proposition~\ref{prop-cube-basis-ii},
 so we may assume that $m<\infty$, so $B'C[n,m]$ is a basis for
 $C[n,m]$ by Proposition~\ref{prop-cube-basis-i}.  However,
 Lemma~\ref{lem-y-powers} implies that $B'C[n,m]$, considered as a
 system of elements in $C[n,m]$, is just the same as $BC[n,m]$.
\end{proof}

\begin{proposition}\label{prop-ann-x}
 Suppose that $0\leq n<k\leq m\leq\infty$ and $k<\infty$.  Then
 $\ann_{C[n,m]}(x_k)=C[n,m]y_{k-1}$. 
\end{proposition}
\begin{proof}
 The $m=\infty$ case will follow from the $m<\infty$ case, because 
 $C[n,m]_d=C[n,\infty]_d$ when $m$ is large relative to $d$.  We will
 thus assume that $m<\infty$.

 We have already observed that $x_ky_{k-1}=0$, so
 $\ann_{C[n,m]}(x_k)\geq C[n,m]y_{k-1}$, and multiplication by $x_k$
 gives a well-defined map $f\:C[n,m]/(C[n,m]y_{k-1})\to C[n,m]$.  It
 will suffice to show that $f$ is injective.

 For this, we put 
 \begin{align*}
  N &= \{\al\in MC[n,m] \st \al_{k-1} = 0\} \\
  A &= \{y^\al\st \al\in N\} \sse C[n,m] \\
  Z &= \text{span}(A) \leq C[n,m].
 \end{align*}
 By inspecting the generators and relations on both sides, we see that
 \[ C[n,m]/(C[n,m]y_{k-1})=\bC[n,k-1]\ot C[k,m]. \]
 Propositions~\ref{prop-cube-basis-i} and~\ref{prop-cube-basis-ii} show that
 $A$ also gives a basis for $C[n,m]/(C[n,m]y_{k-1})$, so
 $C[n,m]=Z\oplus(C[n,m]y_{k-1})$.  Now let $g$ denote the composite 
 \[ Z \xra{\simeq} C[n,m]/(C[n,m]y_{k-1}) \xra{f} C[n,m]
     \xra{\text{proj}} C[n,m]/Z.
 \]
 It will certainly be enough to show that $g$ is injective.  It is not
 hard to see that $y_kZ\leq Z$, and $x_k=y_{k-1}^2+y_k$, so
 $g(z)=x_kz+Z=y_{k-1}^2z+Z$, so $g$ gives an injective map from $A$ to
 $BC[n,m]\sm A$.  These sets are bases for the domain and codomain of
 $g$, so $g$ is injective as required.
\end{proof}

\begin{proposition}\label{prop-cube-selfinj}
 $C$ is self-injective.
\end{proposition}
\begin{proof}
 As $C$ is finite in each degree, it will suffice (by
 Propositions~\ref{prop-finite-baer} and~\ref{prop-block}) to show
 that every test pair $(u,v)$ in $C$ has either a block or a
 transporter.  Let $d$ be the degree of $(u,v)$, so $|v_i|=|u_i|+d$.
 Note that some of the entries $u_i$ and $v_i$ may be zero, in which
 case $|u_i|$ or $|v_i|$ can be negative.  Choose $m$ such that
 $2^m>d$ and also $2^m>|u_i|$ and $2^m>|v_i|$ for all $i$.  Now
 $(u,v)$ can be regarded as a test pair in $C[n,m]$.  Let $\pi$ be the
 projection $C[n,m]\to\bC[n,m]=C[n,m]/y_m$.  As $\bC[n,m]$ has
 Poincar\'e duality, it is self-injective, so the test pair
 $(\pi(u),\pi(v))$ has either a block or a transporter.  First,
 suppose that there is a transporter $\pi(t)$, so $\pi(v_i)=\pi(tu_i)$
 for all $i$.  This is an equation between elements of degree
 $|v_i|<2^m$, and $\pi\:C[n,m]\to\bC[n,m]$ is an isomorphism in this
 degree, so $v_i=tu_i$, so we have a transporter for the original pair
 $(u,v)$.  

 Suppose instead that there is a block for $(\pi(u),\pi(v))$, say
 $\pi(b)$.  This means that $\pi(b.u)=0$ but $\pi(b.v)\neq 0$, so
 $b.u\in C[n,m]y_m$ but $b.v\not\in C[n,m]y_m$.  Using our bases for
 the various rings under consideration, we see that
 $C[n,m]y_m=(Cy_m)\cap C[n,m]$, and thus that $b.v\not\in Cy_m$.  It
 now follows from Proposition~\ref{prop-ann-x} that $(x_{m+1}b).u=0$
 and $(x_{m+1}b).v\neq 0$, so $x_{m+1}b$ is a block for the original
 pair $(u,v)$.
\end{proof}

We now wish to prove that $C$ is coherent, which turns out to involve
substantial work.  It will be convenient to regard the set
$B\bC[n,m]=\{y^\al\st\al\in M\bC[n,m]\}$ as a subset of $C[n,m]$
rather than a subset of $\bC[n,m]$.  We write $\tC[n,m]$ for the span
of this set, so the projection $C[n,m]\to\bC[n,m]$ restricts to give
an isomorphism $\tC[n,m]\to\bC[n,m]$.

\begin{lemma}\label{lem-yyy}
 For $p\geq 3$ we have 
 \[ y_{[0,p-3]}^2 y_{[0,p-1]}^2 y_1y_{p-1}y_p = y_{[0,p]}^2 \]
 (and in particular, this is nonzero modulo $y_{p+1}$).
\end{lemma}
\begin{proof}
 Put $A=C[0,p]/\ann(y_{[0,p]})$.  We claim that in $A$ we have 
 \[ y_{[0,p-3]}^2 y_{[0,p-1]} y_1y_{p-1} = y_{[0,p]}. \]
 Assuming this, we can just multiply by $y_{[0,p]}$ to recover the
 statement in the lemma.

 For $0\leq i<p$ we have $y_i(y_i^2+y_{i+1})=0$ so
 $y_{[0,p]}(y_i^2+y_{i+1})=0$ so $y_{i+1}=y_i^2$ in $A$.  We thus have
 $y_k=y_0^{2^k}$ in $A$ for $0\leq k\leq p$, and so $A=\F[y_0]$.  It
 is thus enough to show that the two sides of the claimed equation
 have the same degree, which is a straightforward calculation.
\end{proof}

\begin{lemma}\label{lem-BB}
 For any $p\geq 3$ we have 
 \[ B\bC[0,p-2]\;B\bC[0,p] \sse
     \coprod_{i=0}^3 B\bC[0,p-1]y_{p-1}^i.
 \]
\end{lemma}
\begin{proof}
 Consider $\al\in M\bC[0,p-2]$ and $\bt\in M\bC[0,p]$.  We note that
 $y^\al,y^\bt\in C[0,p-1]$ so we can rewrite $y^{\al+\bt}$ as an
 element of the basis $B'C[0,p-1]$, which means
 $y^{\al+\bt}=y^\gm$ for some $\gm\in M'C[0,p-1]$.  It will be enough
 to show that $\gm_{p-1}\leq 3$.

 Note that $y^\al$ divides $y_{[0,p-3]}^2$ and $y^\bt$ divides
 $y_{[0,p-1]}^2$ so $y^\gm$ divides $y_{[0,p-3]}^2y_{[0,p-1]}^2$.
 It follows using Lemma~\ref{lem-yyy} that
 $y^\gm y_{p-1}y_p\neq 0\pmod{y_{p+1}}$.  However, 
 \[ y_{p-1}^4y_{p-1}y_p = 
    y_{p-1}^5 y_p = 
    y_{p-1}^3 y_p^2 = 
    y_{p-1} y_p^3 = 
    y_{p-1} y_p y_{p+1} = 0 \pmod{y_{p+1}},
 \]
 so $y^\gm$ cannot be divisible by $y_{p-1}^4$, as required.
\end{proof}

\begin{definition}\label{defn-Kup}
 For any vector $u\in C^n$ and $p\geq 0$, we put 
 \begin{align*}
  K(u,p) &= \{v\in C[0,p]^n\st u.v=0\} \\
  \bK(u,p) &=
    \{ \ov{v} \in \bC[0,p]^n \st \pi(u).\ov{v} = 0\}.
 \end{align*}
 More precisely, $K(u,p)$ is the graded group where 
 \[ K(u,p)_d = \{v\in C[0,p]^n\st |v_i|=d-|u_i|
     \text{ for all } i \text{ and } \sum_iu_iv_i=0\},
 \]
 and $\bK(u,p)$ is graded in a similar way.
\end{definition}

\begin{lemma}\label{lem-K-bK}
 If $u_i\in\tC[0,p-2]$ for all $i$, then the map
 $\pi\:K(u,p+1)\to\bK(u,p+1)$ is surjective.
\end{lemma}
\begin{proof}
 Consider an element $\ov{v}\in\bK(u,p+1)$.  This can be written as
 $\pi(v)$ for a unique element $v\in\tC[0,p+1]^n$, which must satisfy
 $u.v=0\pmod{y_{p+1}}$.  We can write $v$ as
 $\sum_{k=0}^2v_ky_p^k$ with $v_k\in\tC[0,p]^n$.  Using
 Lemma~\ref{lem-BB} we see that $u.v_k$ can be written as
 $\sum_{j=0}^3w_{jk}y_{p-1}^j$ for some elements $w_{jk}\in\tC[0,p-1]$.
 This gives $u.v=\sum_{j=0}^3\sum_{k=0}^2w_{jk}y_{p-1}^jy_p^k$.  After
 reducing the terms $y_{p-1}^jy_p^k$ using the defining relations for
 $C$, we obtain
 \begin{align*}
  u.v =& w_{00} + w_{01}y_p + w_{02}y_p^2 + w_{10}y_{p-1} +
         (w_{11}+w_{30})y_{p-1}y_p + (w_{12}+w_{31})y_{p-1}y_p^2 + \\
       & w_{20}y_{p-1}^2 + w_{21}y_{p-1}^2y_p + w_{22}y_{p-1}^2y_p^2 +
         w_{32}y_{p-1} y_p y_{p+1}.
 \end{align*}
 By hypothesis, this maps to zero in $\bC[0,p+1]=C[0,p+1]/y_{p+1}$.
 However, $\bC[0,p+1]$ splits as the direct sum of subgroups
 $\tC[0,p-1]y_{p-1}^iy_p^j$ for $0\leq i,j<3$, so we must have 
 \[ w_{00} = w_{01} = w_{02} = w_{10} = w_{20} = w_{21} = w_{22} = 0
 \]
 and $w_{11}=w_{30}$ and $w_{12}=w_{31}$, so
 $u.v=w_{32}y_{p-1}y_py_{p+1}$.  
 
 Now put $d=|u.v|$, so $|w_{jk}|=d-j2^{p-1}-k2^p$.  In particular, we
 have $|w_{32}|=d-2^{p-1}-2^p-2^{p+1}$.  If $d<2^{p-1}+2^p+2^{p+1}$
 then $|w_{32}|<0$ so $w_{32}=0$ so $u.v=0$.  This means that
 $v\in K(u,p+1)$ with $\pi(v)=\ov{v}$, as required.  Suppose instead
 that $d\geq 2^{p-1}+2^p+2^{p+1}$.  We have
 \begin{align*}
  |w_{11}| &= |w_{30}| = d - 2^{p-1} - 2^p \geq 2^{p+1} \\
  |w_{12}| &= |w_{31}| = d - 2^{p-1} - 2^{p+1} \geq 2^p.
 \end{align*}
 However, the elements $w_{jk}$ lie in $\tC[0,p-1]$, which is zero in
 degrees larger than $2^p-2$.  We therefore have
 $w_{11}=w_{12}=w_{30}=w_{31}=0$, which means that $u.v_0=0$ and
 $u.v_1=0$ and $u.v_2=w_{32}y_{p-1}^3=w_{32}y_{p-1}y_p$.  Put 
 \[ v' = v_0 + v_1 y_p + v_2 (y_p^2+y_{p+1}), \]
 so $\pi(v')=\pi(v)=\ov{v}$ and 
 \[ u.v' = u.v_0 + u.v_1 y_p + u.v_2(y_p^2+y_{p+1}) =
     w_{32}y_{p-1}y_p(y_p^2+y_{p+1}) = 0.
 \]
 Thus, $v'$ is the required lift of $\ov{v}$ in $\bK(u,p+1)$.
\end{proof}

\begin{lemma}\label{lem-x-poly}
 For all $p\geq 0$ we have a splitting 
 \[ C[0,p+1] = C[0,p] \oplus \bigoplus_{k>0} \bC[0,p]x_{p+1}^k. \]
\end{lemma}
\begin{proof}
 By definition we have $C[0,p+1]=C[0,p][y_{p+1}]/(x_{p+1}y_p)$, where
 $x_{p+1}=y_{p+1}+y_p^2$ as usual.  From this it is clear that 
 \[ C[0,p][y_{p+1}] = C[0,p][x_{p+1}] = 
     C[0,p] \oplus\bigoplus_{k>0} C[0,p]x_{p+1}^k.
 \]
 The ideal generated by $y_px_{p+1}$ in this ring clearly has a
 compatible splitting
 \[ C[0,p][y_{p+1}].y_px_{p+1} = 
     \bigoplus_{k>0} C[0,p]y_p x_{p+1}^k.
 \]
 We can thus pass to the quotient to get 
 \[ C[0,p+1]
    = C[0,p] \oplus \bigoplus_{k>0} \frac{C[0,p]}{C[0,p]y_p} x_{p+1}^k
    = C[0,p] \oplus \bigoplus_{k>0} \bC[0,p] x_{p+1}^k
 \]
 as claimed.
\end{proof}

\begin{corollary}\label{cor-K-step}
 If $u_i\in\tC[0,p-2]$ for $i=0,\dotsc,n-1$, then
 $K(u,p+1)=C[0,p+1].K(u,p)$. 
\end{corollary}
\begin{proof}
 It is clear that $C[0,p+1].K(u,p)\leq K(u,p+1)$.  For the converse,
 consider an element $v\in K(u,p+1)\leq C[0,p+1]^n$.  Using
 Lemma~\ref{lem-x-poly}, we can write $v$ as
 $v_0+\sum_{k>0}\ov{v}_kx_{p+1}^k$, with $v_0\in C[0,p]^n$ and
 $\ov{v}_k\in\bC[0,p]^n$ (with $\ov{v}_k=0$ for $k\gg 0$).  It follows
 that $u.v_0\in C[0,p]$ and $u.\ov{v}_k\in\bC[0,p]$ and
 \[ u.v_0 + \sum_{k>0} (u.\ov{v}_k) x_{p+1}^k = u.v = 0. \]
 As the sum in Lemma~\ref{lem-x-poly} is direct, we must have
 $u.v_0=0$ and $u.\ov{v}_k=0$, so $v_0\in K(u,p)$ and
 $\ov{v}_k\in\bK(u,p)$.  By Lemma~\ref{lem-K-bK}, we can choose
 $v_k\in K(u,p)$ for $k>0$ lifting $\ov{v}_k$.  If $\ov{v}_k=0$ we
 choose $v_k=0$; this ensures that $v_k=0$ for $k\gg 0$.  We now have
 $v=\sum_{k\geq 0}v_kx_{p+1}^k\in C[0,p+1].K(u,p)$, as required.
\end{proof}

\begin{proposition}\label{prop-cube-coherent}
 The ring $C$ is coherent.
\end{proposition}
\begin{proof}
 Let $I\leq C$ be a finitely generated ideal.  Choose elements
 $u_0,\dotsc,u_{n-1}$ generating $I$.  These give an epimorphism
 $g\:\bigoplus_i\Sg^{|u_i|}C\to I$, with $\ker(g)=K(u,\infty)$, so it
 will suffice to show that $K(u,\infty)$ is finitely generated as a
 $C$-module.  Now choose $p$ large enough that $u_i\in\tC[0,p-2]$ for
 all $i$.  As $C[0,p]$ is Noetherian, we can choose a finite subset
 $T\sse C[0,p]^n$ that generates $K(u,p)$ as a $C[0,p]$-module.
 Corollary~\ref{cor-K-step} tells us that $T$ also generates
 $K(u,p+1)$ as a $C[0,p+1]$-module.  In fact, we can apply the same
 corollary inductively to see that $T$ generates $K(u,q)$ as a
 $C[0,q]$-module for all $q\geq p$.  As $C=\bigcup_qC[0,q]$ we
 conclude that $T$ generates $K(u,\infty)$ as required.
\end{proof}

\begin{proposition}\label{prop-cube-reduced}
 The reduced quotient of $C$ is 
 \[ C/\sqrt{0} = \F[x_i\st i\geq 0]/(x_ix_j\st i\neq j). \] 
\end{proposition}
\begin{proof}
 Put $C'=C/\sqrt{0}$.  We first claim that for all $p,q$ with
 $0\leq p<q$ we have $x_px_q=0$ in $C'$.  We may assume inductively
 that $x_ix_j=0$ in $C'$ whenever $0\leq i<j<q$.  By a nested downward
 induction over $p$, we may assume that $x_kx_q=0$ in $C'$ whenever
 $p<k<q$.  As in Proposition~\ref{prop-cube-x}, we have
 $x_q\sum_{k=0}^{q-1}x_k^{2^{q-1-k}}=0$.  We can multiply this by
 $x_p$ and use the inner and outer inductive assumptions to see that
 $x_px_qx_p^{2^{q-1-p}}=0$, or in other words $x_p^mx_q=0$ for some
 $m>0$.  This gives $(x_px_q)^m=0$ in $C'$, but $C'$ is reduced by
 construction so $x_px_q=0$ in $C'$ as claimed.

 Now put 
 \[ C'' = C/(x_ix_j\st i,j,\;i<j) =
     \F[x_i\st i\geq 0]/(x_ix_j\st i,j\geq 0,\;i<j).
 \] 
 We now see that $C''$ is a quotient of $C$ by nilpotent elements, so
 $C'$ can also be described as $C''/\sqrt{0}$.  However, there is an
 obvious splitting 
 \[ C'' = \F \oplus \bigoplus_{i\geq 0} x_i\F[x_i], \]
 and using this we see that $C''$ is reduced.  It follows that
 $C'=C''$ as claimed.
\end{proof}

\section{Pontrjagin self-dual rings}
\label{sec-J}

Let $R$ be a Pontrjagin self-dual ring, as in
Definition~\ref{defn-pontrjagin}.  Thus, $R$ is a graded
$\Zp$-algebra $R$ equipped with an isomorphism $\zt\:R_d\to\Qp/\Zp$
(for some $d$) such that the resulting maps
\[ \zt^\#\:R_{d-k}\to R_k^\vee=\Hom_{\Zp}(R_k,\Qp/\Zp) \]
are isomorphisms.  

\begin{lemma}\label{lem-two-duals}
 For graded $R$-modules $M$ there is a natural isomorphism
 \[ \Hom_R(M,R) \simeq \Hom_{\Zp}(M_d,\Qp/\Zp) = M_d^\vee. \]
\end{lemma}
\begin{proof}
 Given $\phi\in\Hom_R(M,R)$, we put 
 \[ \tau(\phi) = \zt\circ\phi_d\:M_d \to \Qp/\Zp. \]
 This defines a map $\tau\:\Hom_R(M,R)\to M_d^\vee$.  

 Now suppose we have a map $\psi\:M_d\to\Qp/\Zp$.  For any $k\in\Z$
 we have a map 
 \[ \phi'_k\:M_k\to\Hom_{\Zp}(R_{d-k},\Qp/\Zp) \]
 given by $\phi'_k(m)(a)=(-1)^{k(d-k)}\psi(am)$.  As $R$ is assumed to
 be Pontrjagin self-dual, there is a unique element $\phi_k(m)\in R_k$
 such that  
 \[ \phi'_k(m)(a) = \zt(\phi_k(m)a) \]
 for all $a\in R_{d-k}$.  We leave it to the reader to check that
 this gives a map $\phi\:M\to R$ of $R$-modules, and that this is
 the unique such map with $\tau(\phi)=\psi$.
\end{proof}

\begin{proposition}\label{prop-pontrjagin}
 Any Pontrjagin self-dual ring is self-injective.
\end{proposition}
\begin{proof}
 We need to show that the functor $M\mapsto\Hom_R(M,R)$ is
 exact, but it is isomorphic to the functor
 $M\mapsto\Hom_{\Zp}(M_d,\Qp/\Zp)$, which is exact because
 $\Qp/\Zp$ is divisible and therefore injective as a $\Zp$-module.
\end{proof}

We now study the graded ring $J$ described by Definition~\ref{defn-J-ring},
and the tensor product $\hJ=\Zp\ot J$.  It is standard that
$\Zp\ot\Z/p^r=\Z/p^r$.  Moreover, the group $\Qp/\Zp$ can be written
as the colimit of the evident sequence
\[ \Z/p \xra{} \Z/p^2 \xra{} \Z/p^3 \xra{} \dotsc, \]
and we can tensor with $\Zp$ to get $\Zp\ot(\Qp/\Zp)=\Qp/\Zp$.
Thus, the only difference between $J$ and $\hJ$ is that $J_0=\Zpl$
whereas $\hJ_0=\Zp$.

\begin{lemma}\label{lem-J-perfect}
 The ring $\hJ$ is Pontrjagin self-dual, so
 $\hJ_{-2-k}\simeq\hJ_k^\vee$.  
\end{lemma}
\begin{proof}
 For $k\neq -2$ this is a straightforward calculation.  For $k=-2$ we
 use the description $\Qp/\Zp=\colim_j\Z/p^j$ to get 
 \[ \Hom(\Qp/\Zp,\Qp/\Zp)
    = \invlim_j\Hom(\Z/p^j,\Qp/\Zp)
    = \invlim_j \Z/p^j 
    = \Zp,
 \]
 as required.
\end{proof}

\begin{corollary}\label{cor-hJ-selfinj}
 The ring $\hJ$ is self-injective. \qed
\end{corollary}

\begin{remark}\label{rem-J-not-selfinj}
 The ring $J$ itself is not self-injective.  To see this, note that
 $J_{-2}$ is an ideal in $J$ and is a module over $\Zp$.  Choose any
 element $a\in\Zp\sm\Zpl$ and define $u\:J_{-2}\to J$ by $u(x)=ax$.
 This cannot be extended to give a $J$-linear endomorphism of $J$.
\end{remark}

\begin{lemma}\label{lem-hJ-local}
 The ring $\hJ$ is local (in the graded sense).  The unique
 maximal graded ideal is given by $\mxi_0=p\Zp$ and $\mxi_k=\hJ_k$
 for all $k\neq 0$.  Moreover, the elements $\al_k$ together with the
 element $p$ give a basis for $\mxi/\mxi^2$ over $\Z/p$.
\end{lemma}
\begin{proof}
 It is straightforward to check that the graded group $\mxi$ described
 above is an ideal in $\hJ$, and the quotient $\hJ/\mxi$ is the field
 $\Z/p$, so it is a maximal ideal.  Let $\mxi'$ be an arbitrary
 maximal graded ideal.  Put $\mathfrak{a}=\bigoplus_{k\neq 0}\hJ_k$.
 Every homogeneous element $a\in\mathfrak{a}$ satisfies $a^2=0$, and
 it follows that $\mathfrak{a}\leq\mxi'$  This means that $\mxi'$
 corresponds to a maximal ideal in the quotient
 $\hJ/\mathfrak{a}\simeq\Zp$, and the only such ideal is $p\Zp$.  It
 follows that $\mxi'=\mxi$ as claimed.  The description of
 $\mxi/\mxi^2$ is a straightforward calculation.
\end{proof}

\begin{proposition}\label{prop-hJ-incoherent}
 The ring $\hJ$ is totally incoherent.  
\end{proposition}
\begin{proof}
 Put $V=\{\al_k\st k\neq 0\pmod{p}\}\subset J$, so $V$ is infinite and
 $pV=0$ and $V$ and remains linearly independent in $\mxi/\mxi^2$.  By
 inspecting the multiplication rules, we see that every non-invertible
 element of $\hJ$ annihilates all elements of $V$ with at most one
 exception.  It follows using Corollary~\ref{cor-incoherent} that
 $\hJ$ is totally incoherent.
\end{proof}

\section{The infinite root algebra}
\label{sec-root}

In this section we fix a field $K$ and study the infinite root algebra
$P$ over $K$, which was introduced in Definition~\ref{defn-root}.  We
first recall the details.  

\begin{definition}\label{defn-well-ordered}
 We say that a subset $U\sse[0,1]$ is \emph{well-ordered} if the
 usual order inherited from $\R$ is a well-ordering, so every nonempty
 subset of $U$ has a smallest element.  It is equivalent to say that
 every infinite nonincreasing sequence in $U$ is eventually constant, or
 that there are no infinite, strictly decreasing sequences.

 An \emph{infinite root series} is a function
 $a\:[0,1]\to K$ such that the set $\supp(a)=\{q\st a(q)\neq 0\}$
 is well-ordered.  The infinite root algebra is the set $P$ of all
 infinite root series.  We regard this as an ungraded object, or
 equivalently as a graded object concentrated in degree zero.
\end{definition}

\begin{remark}\label{rem-root-addition}
 It is clear that any subset of a well-ordered set is well-ordered,
 and that the union of any two well-ordered sets is well-ordered.  Now
 if $a,b\in P$ we have $\supp(a+b)\sse\supp(a)\cup\supp(b)$, so
 $P$ is closed under addition.  It is clearly also closed under
 multiplication by elements of $K$.
\end{remark}

\begin{lemma}\label{lem-root-order-type}
 Any well-ordered subset of $[0,1]$ is countable.  Moreover, for any
 countable ordinal $\al$, there is a well-ordered subset $U\sse[0,1]$
 that is order-isomorphic to $\al$.
\end{lemma}
\begin{proof}
 Firstly, we can regard rational numbers in $[0,1]$ as coprime pairs
 of integers and this gives a lexicographic ordering on $\Q\cap[0,1]$,
 which is a well-ordering.

 Next, let $U$ be a well-ordered subset of $[0,1]$.  We define
 $f\:U\to\Q$ as follows.  If $u$ is maximal in $U$, we put $f(u)=1$.
 Otherwise, the set $\{v\in U\st v>u\}$ has a smallest element $v_0$,
 and we define $f(u)$ to be the lexicographically smallest element of
 $\Q\cap[u,v_0)$.  It is clear that $f$ is injective, so $U$ is
 countable.  

 Let $\al$ be any countable ordinal; we claim that there is an
 order-embedding $g\:\al\to [0,1]$.  To see this, choose an
 injective map $p\:\al\to\N$ and then put 
 \[ g(\bt)=\sum_{\gm<\bt} 2^{-p(\gm)-1}. \]
 It is clear that this has the required properties.
\end{proof}

\begin{lemma}\label{lem-root-finite}
 If $U,V\sse[0,1]$ are well-ordered and $w\in[0,1]$ then
 $\{(u,v)\in U\tm V\st u+v=w\}$ is finite. 
\end{lemma}
\begin{proof}
 Put $U'=\{u\in U\st w-u\in V\}$.  This is well-ordered (because it is
 a subset of $U$) and it will suffice to show that it is finite.  If
 not, we can define an infinite sequence $u_0<u_1<u_2<\dotsb$ in $U'$
 as follows: we take $u_0$ to be the smallest element in $U'$, then
 take $u_1$ to be the smallest element in $U'\sm\{u_0\}$, and so on.
 We then note that $w-u_0,w-u_1,w-u_2,\dotsc$ is an infinite strictly
 decreasing sequence in $V$, contradicting the assumption that $V$ is
 well-ordered.
\end{proof}

\begin{lemma}\label{lem-subseq}
 Let $U$ be a well-ordered subset of $[0,1]$, and let $(u_n)$ be a
 sequence in $U$.  Then there exists an infinite nondecreasing
 subsequence. 
\end{lemma}
\begin{proof}
 Put $v_0=\min\{u_j\st j\geq 0\}$ (which is meaningful because $U$ is
 well-ordered) and then $n_0=\min\{j\st u_j=v_0\}$.  For $i>0$ we
 define recursively $v_i=\min\{u_j\st j>n_{i-1}\}$ and
 $n_i=\min\{j>n_{i-1}\st u_j=v_i\}$.  We find that $n_0<n_1<n_2<\dotsb$ and
 $v_0\leq v_1\leq v_2\leq\dotsb$, or equivalently 
 $u_{n_0}\leq u_{n_1}\leq u_{n_2}\leq\dotsb$ as required.
\end{proof}

\begin{lemma}\label{lem-convolution}
 Let $U$ and $V$ be well-ordered subsets of $[0,1]$, and put
 $U*V=\{u+v\st u\in U \text{ and } v\in V\}$.  Then $U*V$ is also
 well-ordered.   
\end{lemma}
\begin{proof}
 Suppose not.  We can then find an infinite strictly descending chain
 in $U*V$, so we can choose a sequence $(u_n,v_n)$ in $U\tm V$ with
 $u_i+v_i>u_{i+1}+v_{i+1}$ for all $i$.  Lemma~\ref{lem-subseq} tells
 us that after passing to a subsequence, we may assume that
 $u_j\leq u_{j+1}$ for all $j$.  After passing again to a sparser
 subsequence, we may also assume that $v_k\leq v_{k+1}$ for all $k$.
 This is clearly impossible.
\end{proof}

\begin{proposition}\label{prop-root-ring}
 We can make $P$ into a commutative ring by the rule
 \[ ab(w) = \sum_{w=u+v} a(u)b(v). \]
\end{proposition}
\begin{proof}
 Lemma~\ref{lem-root-finite} shows that the sum is essentially finite,
 so there is no problem with convergence.  It is clear that
 $\supp(ab)\sse\supp(a)*\supp(b)$, and Lemma~\ref{lem-convolution}
 shows that $\supp(a)*\supp(b)$ is well-ordered, so $ab\in P$.  It is
 straightforward to check that the multiplication operation is
 commutative, associative and bilinear.  Moreover, if we define
 $e(0)=1$ and $e(q)=0$ for $q\neq 0$, then $e$ is a multiplicative
 identity element for $P$. 
\end{proof}

\begin{definition}\label{defn-root-dl}
 For $a\in P\sm\{0\}$, we put $\dl(a)=\min(\supp(a))$.  We also put
 $\dl(0)=\infty$.  
\end{definition}
\begin{remark}\label{rem-root-dl}
 Note that if $\dl(a)+\dl(b)\leq 1$ we have 
 \[ (ab)(\dl(a)+\dl(b)) =
     a(\dl(a))\;b(\dl(b)) \neq 0, 
 \]
 so $ab\neq 0$ and $\dl(ab)=\dl(a)+\dl(b)$.  On the other hand, if
 $\dl(a)+\dl(b)>1$ then $ab=0$.
\end{remark}

\begin{definition}\label{defn-x-powers}
 For $q\in\R\cup\{\infty\}$ with $q\geq 0$, we
 define $x^q\in P$ by 
 \[ x^q(u) =
     \begin{cases}
      1 & \text{ if } u=q \\
      0 & \text{ otherwise. }
     \end{cases}
 \]
\end{definition}

\begin{remark}\label{rem-x-powers}
 We note that
 \begin{itemize}
  \item[(a)] $x^0$ is the multiplicative identity element $e$.
  \item[(b)] If $q>1$ then $x^q=0$.
  \item[(c)] If $0\leq q \leq 1$ then $\dl(x^q)=q$.
  \item[(d)] For all $q,r\geq 0$ we have $x^qx^r=x^{q+r}$.
 \end{itemize}
\end{remark}

\begin{lemma}\label{lem-root-units}
 Consider an element $a\in P\sm\{0\}$.  If $a(0)=0$ (or equivalently,
 $\dl(a)>0$) then $a$ is nilpotent, but if $\dl(a)=0$ then $a$ is
 invertible. 
\end{lemma}
\begin{proof}
 If $\dl(a)>0$ then we can find a positive integer $n$ with
 $\dl(a)>1/n$, and using Remark~\ref{rem-root-dl} we see that
 $a^n=0$.  Suppose instead that $\dl(a)=0$.  We can then write
 $a=ue+b=u(e+b/u)$ where $u\in K\sm 0$ and $e=x^0$ is the
 multiplicative identity of $P$ and $\dl(b)>0$, so $b^n=0$ for some
 $n$.  Now $a$ has inverse $\sum_{i=0}^{n-1} u^{-1}(-b/u)^i$.
\end{proof}

\begin{corollary}\label{cor-root-reduced}
 The map $a\mapsto a(0)$ induces an isomorphism $P/\sqrt{0}\to K$.
\end{corollary}
\begin{proof}
 Clear.
\end{proof}

\begin{definition}\label{defn-root-lambda}
 For $a\in P$ with $\dl(a)\geq t$, we define $\lm_t(a)\in P$ by 
 \[ \lm_t(a)(r) = 
     \begin{cases}
      a(r+t) & \text{ if } 0 \leq r \leq 1-t \\
      0 & \text{ if } 1-t < r \leq 1. 
     \end{cases}
 \]
\end{definition}

\begin{corollary}\label{cor-root-units}
 If $\dl(a)\geq t$ then $a=x^t\,\lm_t(a)$ and
 $\dl(\lm_t(a))=\dl(a)-t$.  Moreover, if $\dl(a)=t$ then $\lm_t(a)$ is
 invertible, so $Pa=Px^t$.
\end{corollary}
\begin{proof}
 The first two claims are clear from the definitions, and the third
 then follows using Lemma~\ref{lem-root-units}.
\end{proof}

\begin{definition}\label{defn-root-Jt}
 For $t\in[0,1]$ we put 
 \begin{align*}
  J_t &= \{a\in P \st\dl(a)>t\} \\
  \ov{J}_t &= \{a\in P\st\dl(a)\geq t\} = Px^t.
 \end{align*}
\end{definition}

\begin{proposition}\label{prop-root-ideals}
 Every ideal in $P$ has the form $J_t$ or $\ov{J}_t$.
\end{proposition}
\begin{proof}
 Let $I$ be an ideal in $P$.  If $I=0$ then $I=J_1$.  Otherwise, we
 put $t=\inf\{\dl(a)\st a\in I\}$.  If $t=\dl(a)$ for some $a\in I$
 then Corollary~\ref{cor-root-units} shows
 that $x^t\in I$, and it follows easily that $I=\ov{J}_t$.  Suppose
 instead that there is no element $a\in I$ with $\dl(a)=t$.  It is
 then clear that $I\leq J_t$.  Moreover, if $b\in J_t$ then $\dl(b)>t$
 so (by the infimum condition) there exists $a\in I$ with
 $\dl(b)>\dl(a)>t$.  After applying Corollary~\ref{cor-root-units} to
 $a$ and $b$, we see that $b$ is a multiple of $a$, and so $b\in I$.
 We now see that $I=J_t$, as required.
\end{proof}

\begin{proposition}\label{prop-root-ann}
 For all $t\in [0,1]$ we have $\ann_P(J_t)=\ov{J}_{1-t}$ and
 $\ann_P(\ov{J}_t)=J_{1-t}$. 
\end{proposition}
\begin{proof}
 This follows easily from the fact that $ab=0$ iff $\dl(a)+\dl(b)>1$. 
\end{proof}

\begin{corollary}
 For any ideal $I\leq P$ we have $\ann^2_P(I)=I$.
\end{corollary}
\begin{proof}
 Immediate from the last two propositions.
\end{proof}

\begin{proposition}\label{prop-root-selfinj}
 $P$ is self-injective.
\end{proposition}
\begin{proof}
 As we have classified all ideals in $P$, we can use Baer's criterion.
 Consider a number $t\in [0,1]$ and a $P$-module map
 $f\:\ov{J}_t=(x^t)\to P$.  If $f(x^t)=a$ then we must have
 $J_{1-t}a=f(J_{1-t}x^t)=f(0)=0$, so $a\in\ann(J_{1-t})=\ov{J}_t$, so
 $a=x^t\lm_t(a)$.  We can now define $f'\:P\to P$ extending $f$ by
 $f'(p)=p\,\lm_t(a)$, so Baer's criterion is satisfied in this case.

 Now consider instead a $P$-module map $f\:J_t\to P$.  If $t=1$ then
 $J_t=0$ and the zero map $P\to P$ extends $f$.  We suppose instead
 that $t<1$.  For $s\in(t,1]$ we put $a_s=\lm_s(f(x^s))$, so the first
 case shows that $f(p)=p a_s$ for all $p\in\ov{J}_s<J_t$.  Now suppose
 that $t<r\leq s\leq 1$.  As $x^s\in\ov{J}_s\leq\ov{J}_r$ we have
 $x^s(a_r-a_s)=f(x^s)-f(x^s)=0$, so $a_r(q)=a_s(q)$ for all
 $q\leq 1-s$.  Moreover, from the definition of the $\lm$ operation we
 have $a_s(q)=0$ for $q>1-s$, and thus certainly for $q\geq 1-t$.  We
 now see that there is a unique map $a\:[0,1]\to K$ with $a=a_s$ on
 $[0,1-s]$ (for all $s\in(t,1]$) and $a=0$ on $[1-t,1]$.  It
 follows easily from these properties that $\supp(a)$ is well-ordered,
 so $a\in P$.  We also see from the first property that $f$ agrees
 with multiplication by $a$ on $\ov{J}_s$ for all $s\in(t,1]$.  It
 follows that the same is true on $\bigcup_{s\in(t,1]}\ov{J}_s=J_t$,
 as required.
\end{proof}

\begin{proposition}\label{prop-root-incoherent}
 $P$ is totally incoherent.
\end{proposition}
\begin{proof}
 Let $I$ be a finitely generated ideal, say $I=(a_1,\dotsc,a_r)$,
 where we can assume that the generators $a_i$ are nonzero.  If $r=0$
 then $I=0$, and this is finitely presented.  If $r>0$ we can use
 Corollary~\ref{cor-root-units} to see that $I=\ov{J}_t$, where
 $t=\min(\dl(a_1),\dotsc,\dl(a_r))$.  

 Now suppose that $I$ is nonzero and finitely presented.  We must
 have $I=\ov{J}_t$ for some $t$, so we have an epimorphism $g\:P\to I$
 given by $g(a)=ax^t$.  Definition~\ref{defn-fp} tells us that
 $\ker(g)$ must also be finitely generated, but
 $\ker(g)=\ann_P(x^t)=J_{1-t}$, and this is only finitely generated
 when $t=0$ and so $\ker(g)=J_1=0$ and $I=\ov{J}_0=P$.
\end{proof}

\begin{remark}\label{rem-P-prime}
 Put $P'=\{a\in P\st\supp(a)\sse\Q\}$.  This is a subring of $P$, and
 one can adapt the above arguments to show that it is again
 self-injective and totally incoherent.  Every ideal in $P'$ has the
 form $J_t\cap P'$ or $\ov{J}_t\cap P'$ for some $t\in[0,1]$, and
 these are all distinct except for the fact that
 $J_t\cap P'=\ov{J}_t\cap P'$ when $t$ is irrational.
\end{remark}

\section{The Rado algebra}
\label{sec-rado}

In this section we study the Rado algebra $Q$, which was defined in
Definition~\ref{defn-rado}.  We will write $\Gm$ for the Rado graph.

We first clarify the kinds of graphs that we will consider.
\begin{definition}\label{defn-graph}
 A \emph{graph} is a pair $(V,E)$, where $V$ is a set and $E$ is a
 subset of $V\tm V$ such that
 \begin{itemize}
  \item[(a)] For all $v\in V$ we have $(v,v)\not\in E$.
  \item[(b)] For all $v,w\in V$ we have ($(v,w)\in E$ iff
   $(w,v)\in E$). 
 \end{itemize}
\end{definition}
\begin{definition}\label{defn-graph-embedding}
 Let $G=(V,E)$ and $G'=(V',E')$ be graphs.  A \emph{full embedding} of
 $G$ in $G'$ is an injective map $f\:V\to V'$ such $E=(f\tm
 f)^{-1}(E')$ (so vertices $v_0,v_1\in V$ are linked by an edge in $G$
 iff the images $f(v_0)$ and $f(v_1)$ are linked by an edge in $G'$).
 Similarly, a \emph{full subgraph} of $G'$ is a graph of the form
 $G=G'|_V=(V,E'\cap V^2)$ for some subset $V\sse V'$, so the inclusion
 map gives a full embedding $G\to G'$.
\end{definition}

\begin{lemma}\label{lem-graph-extension}
 Suppose we have a finite graph $G'$, a full subgraph $G$, and a full
 embedding $f\:G\to\Gm$.  Then there is a full embedding
 $f'\:G'\to\Gm$ extending $f$.
\end{lemma}
\begin{proof}
 It is easy to reduce to the case where $G'$ has only one more vertex
 than $G$, say $V'=V\amalg\{x\}$.  Put $A=\{v\in V\st (v,x)\in E'\}$
 and $N=\max\{f(v)\st v\in V\}+1$, then let $f'\:V'\to\N$ be the map
 extending $f$ with $f'(x)=2^N+\sum_{v\in A}2^{f(v)}$.  It is
 straightforward to check that this has the required properties.
\end{proof}

\begin{remark}\label{rem-rado-thin}
 As we mentioned in Example~\ref{eg-exterior}, each group $E_k$ (for
 $k\geq 0$) is isomorphic to $\F$.  The generator is the element
 $y_k=x_{B(k)}=\prod_{i\in B(k)}x_i$.  We say that a finite subset
 $I\sse\N$ is \emph{$\Gm$-complete} if the full subgraph $\Gm|_I$ is a
 complete graph (so every two distinct points are linked by an edge).
 We say that a natural number $n$ is \emph{$B\Gm$-complete} if $B(n)$
 is $\Gm$-complete.  It is clear that the set 
 \[ \{y_n\st n \text{ is not $B\Gm$-complete } \} \]
 is a basis for the Rado ideal, and thus that the set 
 \[ \{y_n\st n \text{ is $B\Gm$-complete } \} \]
 gives a basis for $Q$.
\end{remark}

\begin{proposition}\label{prop-rado-dac}
 For any finitely generated ideal $I\leq Q$, we have $\ann^2(I)=I$.
 (In other words, $Q$ satisfies the double annihilator condition.)
\end{proposition}
\begin{proof}
 Let $I\leq Q$ be a finitely generated ideal.  Because of
 Remark~\ref{rem-rado-thin}, the ideal $I$ must be generated by a
 finite list of monomials, say $I=(x_{A_1},\dotsc,x_{A_r})$, where
 each $A_i$ is a finite $\Gm$-complete subset of $\N$.  Simiilarly,
 $\ann^2(I)$ is generated by the monomials that it contains.

 Let $T$ be another $\Gm$-complete subset of $\N$.  If $T$ contains
 $A_i$ for some $i$, it is clear that $x_T\in I$.  Suppose instead
 that $T$ does not contain any of the $A_i$.  Let $N$ be strictly
 larger than any of the elements of $T\cup\bigcup_iA_i$, and put
 $n=2^N+\sum_{t\in T}2^t$, so $B(n)=\{N\}\cup T$.  It is clear that
 $n\not\in T$ and $T\cup\{n\}$ is $\Gm$-complete so $x_nx_T\neq 0$.
 However, we claim that $x_nx_{A_i}=0$ for all $i$.  Indeed, as
 $T\not\supseteq A_i$ we can choose $k\in A_i\sm T$.  As $N$ is so
 large we cannot have $n\in B(k)$, and also
 $k\not\in\{N\}\cup T=B(n)$, so $x_nx_k=0$, so $x_nx_{A_i}=0$ as
 claimed.  We now see that $x_n\in\ann(I)$, but $x_nx_T\neq 0$, so
 $x_T\not\in\ann^2(I)$.  It follows that $\ann^2(I)=I$ as claimed.
\end{proof}

\begin{proposition}\label{prop-rado-not-selfinj}
 $Q$ is not self-injective.  
\end{proposition}
\begin{proof}
 Take any pair $p,q\in\N$ with $p\neq q$ and $x_px_q=0$ (say $p=0$ and
 $q=2$).  Put $u=(x_p,x_q)$ and $v=(0,x_q)$, and consider the test
 pair $(u,v)$.  Any transporter would have to be an element
 $t\in Q_0=\{0,1\}$ with $tx_p=0$ and $tx_q=x_q$.  It is clear from
 this that there is no transporter.  A block would be a pair $(a,b)$
 with $bx_q\neq 0$ but $ax_p+bx_q=0$ (so $ax_p=bx_q\neq 0$).  This
 means that $a$ and $b$ are nonzero homogeneous elements, say $a=x_A$
 and $b=x_B$ for some $\Gm$-complete sets $A$ and $B$.  As
 $ax_p\neq 0$ we see that $p\not\in A$, and that $A\cup\{p\}$ is again
 $\Gm$-complete.  Similarly, we have $q\not\in B$ and $B\cup\{q\}$ is
 $\Gm$-complete.  The equation $ax_p=bx_q$ means that
 $A\cup\{p\}=B\cup\{q\}$, so we have $A=C\cup\{q\}$ and $B=C\cup\{p\}$
 for some set $C$.  This now gives $bx_q=x_Cx_px_q$ but $x_px_q=0$ so
 $bx_q=0$, contrary to assumption.  This shows that we have neither a
 block nor a transporter, so $Q$ is not self-injective.
\end{proof}

\begin{remark}\label{rem-rado-regrading}
 We could give $Q$ a different grading with such that there are some
 pairs $(i,j)$ with $i\neq j$ but $|x_i|=|x_j|$, so $x_i+x_j$ becomes
 homogeneous.  One can check that if $x_ix_j=0$ then
 $\ann^2(x_i+x_j)=(x_i,x_j)\neq(x_i+x_j)$, so the double annihilator
 condition no longer holds.  We will discuss a similar situation with
 more details in Lemma~\ref{lem-epsilon-not-dac}.  We believe that the
 self-injectivity condition is similarly sensitive to the choice of
 grading, but we do not have an example to prove this.
\end{remark}

\begin{proposition}\label{prop-rado-incoherent}
 $Q$ is totally incoherent.
\end{proposition}
\begin{proof}
 First, it is clear that $Q$ is local, with maximal ideal
 $\mxi=(x_i\st i\in\N)=\bigoplus_{k>0}Q_k$.  The generators $x_i$ form
 a basis for $\mxi/\mxi^2$.  Note that if $A\subset\N$ is nonempty and
 $\Gm$-complete, then infinitely many of the variables $x_i$ will
 satisfy $x_ix_A=0$, so the image of $\ann(x_A)$ in $\mxi/\mxi^2$ will
 have infinite dimension.  The claim therefore follows by
 Corollary~\ref{cor-incoherent}. 
\end{proof}

\section{The $\ep_0$-algebra}
\label{sec-epsilon}

The $\ep_0$ algebra $A$ was introduced in
Definition~\ref{defn-epsilon-intro}.  We now explain the definition in
more detail, and prove some properties.

\begin{definition}
 Suppose we have a sequence $\un{\bt}=(\bt_1>\bt_2>\dotsb>\bt_r)$ of
 ordinals, and a sequence $\un{n}=(n_1,\dotsc,n_r)$ of positive
 integers.   We write 
 \[ C(\un{\bt},\un{n}) =
     \om^{\bt_1}n_1 + \dotsc + \om^{\bt_r}n_r. 
 \]
 Note that this uses ordinal exponentiation, defined in the usual
 recursive way by $\al^{\bt+1}=\al\al^\bt$ and
 $\al^\lm=\bigcup_{\bt<\lm}\al^\bt$ when $\lm$ is a limit ordinal.
\end{definition}

The following fact is standard (and not hard to prove by transfinite
induction). 
\begin{proposition}
 For any ordinal $\al$ there is a unique pair $(\un{\bt},\un{n})$ such
 that $\al=C(\un{\bt},\un{n})$.  (This is the \emph{Cantor normal
  form} for $\al$.)
\end{proposition}
\begin{proof}
 See~\cite{jo:nls}*{Exercise 6.10}, for example.
\end{proof}

\begin{definition}
 We put $\pi_0=\om$ and define $\pi_n$ recursively by
 $\pi_{n+1}=\om^{\pi_n}$, and then put $\ep_0=\bigcup_n\pi_n$.  
\end{definition}

One can check that $\ep_0=\om^{\ep_0}$, and that $\ep_0$ is the
smallest ordinal with this property.  Note that the expression
$\ep_0=\om^{\ep_0}$ is the Cantor normal form of $\ep_0$.  For
$\al<\ep_0$ we find that the exponents $\bt_t$ in the Cantor normal
form of $\al$ are strictly less than $\al$, so in this case one can
do induction or recursion based on the Cantor normal form.

\begin{definition}
 We define $\dl\:\ep_0\to\N$ recursively by $\dl(0)=1$ and 
 $\dl(\al)=(\sum_t(\dl(\bt_t)+2)n_t)-1$ if
 $\al=\om^{\bt_1}n_1+\dotsb+\om^{\bt_r}n_r$. 
\end{definition}

We will give enough examples to show that $\dl$ is not injective,
which will be needed later.
\begin{example}\label{eg-delta}
 \begin{align*}
  \dl(1)     &= \dl(\om^0) = (\dl(0)+2) - 1 = 2 \\
  \dl(2)     &= \dl(\om^0\,2) = (\dl(0)+2)2 - 1 = 5 \\
  \dl(\om)   &= \dl(\om^1) = (\dl(1)+2) - 1 = 3 \\
  \dl(\om+1) &= \dl(\om^1+\om^0) = (\dl(1)+2)+(\dl(0)+2)-1 = 6 \\
  \dl(\om^2) &= (\dl(2)+2) - 1 = 6.
 \end{align*}
\end{example}

In order to analyse $\dl$, it is helpful to modify the Cantor normal
form slightly.
\begin{lemma}
 If $\al<\ep_0$ then there is a unique way to write
 \[ \al = \om^{\bt_1} + \om^{\bt_2} + \dotsb + \om^{\bt_m} \]
 with $\al>\bt_1\geq\bt_2\geq\dotsb\geq\bt_m$.  (This is the
 \emph{expanded Cantor normal form}.)
\end{lemma}
\begin{proof}
 Just take the ordinary Cantor normal form and replace
 $\om^{\bt_t}n_t$ by $n_t$ copies of $\om^{\bt_t}$.
\end{proof}

\begin{lemma}\label{lem-dl-finite}
 For any $d\in\N$ there are only finitely many ordinals $\al\in\ep_0$
 with $\dl(\al)=d$.  
\end{lemma}
\begin{proof}
 Let $A$ denote the alphabet $\{0,\pi,+\}$.  For each $\al<\ep_0$ we
 define a word $\phi(\al)$ in $A$ as follows.  We start with
 $\phi(0)=0$.  If $\tht>0$ has expanded Cantor normal form 
 $\tht=\om^{\bt_1}+\dotsb+\om^{\bt_m}$ we put 
 \[ \phi(\tht) =
     \phi(\bt_1)\pi\phi(\bt_2)\pi\dotsb\phi(\bt_m)\pi+\dotsb+
 \]
 (with $m-1$ plusses at the end).  For example we have
 \begin{align*}
  \phi(3) &= \phi(\om^0+\om^0+\om^0) = 0\pi 0\pi 0\pi ++ \\
  \phi(\om^\om+\om) &= 0\pi\pi\pi 0\pi\pi+.
 \end{align*}
 It is clear from the definitions that $\dl(\tht)$ is the length of
 $\phi(\tht)$, and there are only $3^d$ words in $A$ of length $d$, so
 it will suffice to show that $\phi$ is injective.  If we interpret
 $\pi$ as the operator $x\mapsto\om^x$ then $\phi(\tht)$ is a reverse
 polish expression that evaluates to $\tht$, and this implies
 injectivity.
\end{proof}

\begin{corollary}
 $\ep_0$ is countable. \qed
\end{corollary}

\begin{definition}
 Let $\tA$ be the graded polynomial algebra over $\F$ generated by
 elements $x_\al$ for each ordinal $\al<\ep_0$, with
 $|x_\al|=\dl(\al)$.  
\end{definition}

Using Lemma~\ref{lem-dl-finite} we see that $\tA_d$ is finite for all
$d$. 

\begin{definition}\label{defn-mu}
 For ordinals $\al,\bt<\ep_0$ with $\al\neq\bt$ we define
 $\mu_0(\al,\bt)$ to be the coefficient of $\om^\bt$ in $\al$.  More
 explicitly, if the Cantor normal form of $\al$ involves a term
 $\om^\bt n$, then $\mu_0(\al,\bt)=n$; if there is no such term then
 $\mu_0(\al,\bt)=0$.  One can check that if $\mu_0(\al,\bt)>0$ then
 $\mu_0(\bt,\al)=0$.  We put
 $\mu(\al,\bt)=\max(\mu_0(\al,\bt),\mu_0(\bt,\al))$. 
\end{definition}

\begin{proposition}\label{prop-mu-generic}
 For any finite set $J\subset\ep_0$ and map $\nu\:J\to\N$ there exists
 $\al\in\ep_0\sm J$ such that $\mu(\al,\bt)=\nu(\bt)$ for all
 $\bt\in J$.   (We will call this the \emph{extension property}.)
\end{proposition}
\begin{proof}
 Write $J$ in order as $J=\{\bt_1>\bt_2>\dotsb>\bt_r\}$ and then take
 \[ \al =
     \om^{\bt_1+1} +
       \om^{\bt_1}.\nu(\bt_1) + \dotsb + \om^{\bt_r}.\nu(\bt_r).
 \]
 It is visible that $\mu_0(\al,\bt_t)=\nu(\bt_t)$ for all $t$.  Also,
 because of the initial term $\om^{\bt_1+1}$ we have
 $\om^\al>\al>\bt_t$ for all $t$ and so $\mu_0(\bt_t,\al)=0$.  It
 follows that $\mu(\al,\bt_t)=\nu(\bt_t)$ for all $t$, as required.
\end{proof}

From now on we will only need the fact that our index set $\ep_0$ is
countable and that the extension property holds.  It will therefore be
notationally convenient to write $I=\ep_0$ and ignore the fact that
the elements of $I$ are ordinals, and to write $i$ instead of $\al$
for a typical element of $I$.  We also put 
$I_2=\{(i,j)\in I^2\st i\neq j\}$.

\begin{definition}\label{defn-epsilon}
 For each $(i,j)\in I_2$ we put $\rho(i,j)=x_ix_j^{\mu(i,j)+1}$.  We
 then let $A$ be the quotient of $\tA$ by all such elements
 $\rho(i,j)$.  We call this the \emph{$\ep_0$-algebra}.
\end{definition}

\begin{definition}\label{defn-epsilon-bases}
 Given a map $\al\:I\to\N$, we write $\supp(\al)=\{i\st\al(i)>0\}$.
 Let $M\tA$ be the set of all such maps $\al$ for which $\supp(\al)$ is
 finite.  For $\al\in M\tA$ we put $x^\al=\prod_ix_i^{\al(i)}\in\tA$.
 We write $B\tA$ for the set of all such monomials $x^\al$, so $B\tA$ is
 a basis for $\tA$.  Next, put 
 \[ MA = \{\al\in M\tA \st \forall i\neq j \;
     \al(i)>0 \Rightarrow \al(j) \leq \mu(i,j)\}
 \]
 and $BA=\{x^\al\st\al\in MA\}$.  One can check that $BA$ gives a basis
 for $A$. 
\end{definition}

\begin{definition}
 A \emph{monomial ideal} is just an ideal in $A$ that is generated by
 some subset of $BA$.
\end{definition}

\begin{remark}\label{rem-monomial-ideal-basis}
 Let $P$ be a monomial ideal, generated by $\{x^\al\st\al\in U\}$ for
 some subset $U\sse MA$.  Put 
 \[ U^+=\{\al\in MA\st \al\geq\bt \text{ for some } \bt\in U\}. \]
 It is easy to see that $\{x^\al\st\al\in U^+\}$ is then a basis for
 $P$ over $\F$.  It follows easily that sums, products, intersections
 and annihilators of monomial ideals are again monomial ideals.
\end{remark}

\begin{lemma}\label{lem-monomial-generators}
 If $P$ is a monomial ideal then it is finitely generated if and only
 if there is a finite list of monomials that generate it.
\end{lemma}
\begin{proof}
 Suppose that $P$ is generated by $a_1,\dotsc,a_m$, where the elements
 $a_t$ need not be monomials.  We can write
 $a_t=\sum_{\al\in U_t}a_{t,\al}x^\al$, for some finite set
 $U_t\subset MA$ and some nonzero coefficients $a_{t,\al}$.  Using
 Remark~\ref{rem-monomial-ideal-basis} we see that the terms $x^{\al}$
 (for $\al\in U_t$) lie in $P$.  Put $U=\bigcup_tU_t$ (which is
 finite) and put $P'=(x^{\al}\st\al\in U)\leq P$.  Clearly
 $a_t\in(x^{\al}\st\al\in U_t)\leq P'$ and the elements $a_t$ generate
 $P$ so $P\leq P'$ so $P=P'$.  Thus, $P$ is generated by a finite list
 of monomials.
\end{proof}

\begin{proposition}\label{prop-epsilon-dac}
 Let $P\leq A$ be a finitely generated monomial ideal.  Then
 $\ann^2(P)=P$.
\end{proposition}
\begin{proof}
 It is automatic that $P\leq\ann^2(P)$, so it will suffice to
 prove the opposite inclusion.  Note that both $P$ and $\ann^2(P)$ are
 monomial ideals, so it will suffice to show that they contain the
 same monomials.  Suppose that $x^\bt$ is a nonzero monomial that does
 not lie in $P$; we must find $y\in\ann(P)$ such that $x^\bt y\neq 0$.

 We can choose a finite list $\al_1,\dotsc,\al_r\in M$ such that
 $P=(x^{\al_1},\dotsc,x^{\al_r})$.  Put
 $J=\supp(\bt)\cup\bigcup_i\supp(\al_i)$, which is a finite subset of
 $I$.  Put $N=\max\{\bt(j)\st j\in J\}$.

 Next, for each $t$ we note that $x^\bt$ cannot be divisible by
 $x^{\al_t}$, so we can choose $i_t\in J$ such that
 $\al_t(i_t)>\bt(i_t)$.  Using the extension property we can
 recursively define distinct elements $k_1,\dotsc,k_r\in I\sm J$ such
 that 
 \begin{itemize}
  \item[(a)] $\mu(k_t,i_t)=\al_t(i_t)-1$
  \item[(b)] $\mu(k_t,j)=N$ for $j\in J\sm\{i_t\}$
  \item[(c)] $\mu(k_t,k_s)=1$ for $s<t$.
 \end{itemize}
 Put $y=\prod_tx_{k_t}$.  This is nonzero by property~(c).
 Property~(a) tells us that $x_{j_t}x^{\al_t}=0$ for all $t$, which
 implies that $y\in\ann(A)$.  On the other hand, we note that
 \begin{itemize}
  \item Clause~(a) above tells us that $y x^\bt$ is not divisible by
   any relator $\rho(k_t,i_t)$.
  \item Clause~(b) tells us that $y x^\bt$ is not divisible by
   any relator $\rho(k_t,j)$ with $j\in J\sm\{i_t\}$.
  \item Clause~(c) tells us that $y x^\bt$ is not divisible by
   any relator $\rho(k_t,k_s)$.
  \item Our original assumption $x^\bt\neq 0$ implies that $y x^\bt$
   is not divisible by any relator $\rho(j,j')$ with $j,j'\in J$.
 \end{itemize}
 This shows that $y x^\bt\neq 0$, but $y\in\ann(P)$, so
 $x^\bt\not\in\ann^2_(P)$, as claimed.
\end{proof}

\begin{lemma}\label{lem-epsilon-not-dac}
 Let $i$ and $j$ be any two distinct indices in $I$ with
 $|x_i|=|x_j|$ and $\mu(i,j)=0$.  Then
 $\ann^2(x_i+x_j)=(x_i,x_j)>(x_i+x_j)$. 
\end{lemma}
\begin{proof}
 As $\mu(i,j)=0$ we have $x_ix_j=0$ and so
 (using monomial bases) $(x_i)\cap(x_j)=0$.  If $u(x_i+x_j)=0$ then we
 have $ux_i=-ux_j$, with the left hand side in $(x_i)$ and the right
 hand side in $(x_j)$.  As $(x_i)\cap(x_j)=0$ this gives
 $ux_i=ux_j=0$.  It now follows that $\ann(x_i+x_j)=\ann(x_i,x_j)$ and
 so $\ann^2(x_i+x_j)=\ann^2(x_i,x_j)$.  As $(x_i,x_j)$ is a monomial
 ideal we also have $\ann^2(x_i,x_j)=(x_i,x_j)$, so
 $\ann^2(x_i+x_j)=(x_i,x_j)>(x_i+x_j)$ as claimed.
\end{proof}

\begin{corollary}\label{cor-epsilon-not-selfinj}
 Example~\ref{eg-delta} shows that the lemma applies to the pair
 $(\om^2,\om+1)$, so $A$ does not satisfy the double annihilator
 condition.  Thus, Remark~\ref{rem-ideal} shows that $A$ cannot be
 self-injective. \qed
\end{corollary}

\begin{remark}\label{rem-epsilon-regrading}
 We could choose a different grading such that all the generators had
 different degrees, which would eliminate any examples as in
 Lemma~\ref{lem-epsilon-not-dac}.  However, we cannot ensure that
 $A_d$ has dimension at most one for all $d$, because when $i\neq j$
 the elements $x_i^{|x_j|}$ and $x_j^{|x_i|}$ have the same degree and
 are linearly independent.  Thus, there will always be ideals that are
 not monomial ideals.  We suspect that there is no grading for which
 $A$ satisfies the full double annihilator condition, but we have not
 proved this.
\end{remark}

\begin{proposition}\label{prop-epsilon-incoherent}
 $A$ is totally incoherent.
\end{proposition}
\begin{proof}
 Put $\mxi_0=0$ and $\mxi_k=A_k$ for all $k>0$, so $A/\mxi=\F$.  It is
 clear that $\mxi$ is an ideal, and that the (homogeneous) elements of
 $\mxi$ are precisely the elements of $A$ that are not invertible.
 Given this, it follows that $\mxi$ is the unique maximal ideal in
 $A$, so $A$ is local.  From the form of the relations in $A$ we see
 that $\{x_i\st i\in I\}$ is a basis for $\mxi/\mxi^2$.  

 Now consider an element $a\in A_d$ for some $d>0$.  Put
 \begin{align*}
  U &= \{i\in I\st\dl(i)\leq d\} \\
  V &= \{\om^i \st i\in I\sm U\}.
 \end{align*}
 We find that $x_ix_j=0$ for all $i\in U$ and $j\in V$.  Moreover, we
 have $a\in(x_i\st i\in U)$, so $ax_j=0$ for all $j\in V$, so the
 image of $\ann(a)$ in $\mxi/\mxi^2$ has infinite dimension.

 Now let $P$ be a finitely presented ideal in $A$.  If
 $P=\mxi P$ then $P=0$ by Nakayama's Lemma.  Otherwise, we can
 choose $a\in P\sm \mxi P$, and Lemma~\ref{lem-ann-fg} tells us that
 $\ann(a)$ has finite image in $\mxi/\mxi^2$.  The above remarks show
 that we must have $|a|=0$, and $a\not\in \mxi P$ so $a\neq 0$, so $a$
 is invertible, so $P=A$.
\end{proof}

\begin{proposition}\label{prop-epsilon-reduced}
 The reduced quotient is  
 \[ A/\sqrt{0} = \F[x_i\st i\in I]/(x_ix_j\st i\neq j). \]
\end{proposition}
\begin{proof}
 In $A$ we have $x_ix_j^{\mu(i,j)+1}=0$, so $(x_ix_j)^{\mu(i,j)+1}=0$,
 so $x_ix_j$ is nilpotent.  If we put 
 \[ A' = A/(x_ix_j\st i\neq j) = 
     \F[x_i\st i\in I]/(x_ix_j\st i\neq j),
 \]
 we deduce that $A/\sqrt{0}=A'/\sqrt{0}$.  However, it is easy to see
 that $A'$ is already reduced, so $A/\sqrt{0}=A'$ as claimed.
\end{proof}

\section{Triangulation}
\label{sec-triangulation}

Recall that a \emph{triangulated category} is a triple
$(\CC,\Sg,\Dl)$, where $\CC$ is an additive category, and
$\Sg\:\CC\to\CC$ is an equivalence, and $\Dl$ is a class of diagrams
of shape 
\[ X \to Y \to Z \to \Sg X \]
(called \emph{distinguished triangles}), subject to certain axioms
that we will not list here.

\begin{definition}\label{defn-triangulation}
 Let $R$ be a self-injective graded ring, let $\Mod_R$ be the category
 of $R$-modules, and let $\Sg\:\Mod_R\to\Mod_R$ be the usual suspension
 functor so that $(\Sg M)_i=M_{i-1}$.  Let $\InjMod_R$ be the full
 subcategory of injective modules.  A \emph{triangulation structure}
 for $R$ is a pair $(\CN,\Dl)$, where
 \begin{itemize}
  \item[(a)] $\CN$ is a full subcategory of $\InjMod_R$ containing $R$.
  \item[(b)] $\CN$ is closed under finite direct sums, retracts,
   suspensions and desuspensions.
  \item[(c)] $\Dl$ is a class of distinguished triangles
   making $(\CN,\Sg,\Dl)$ into a triangulated category.
 \end{itemize}
\end{definition}

We can also make a similar definition for ungraded rings.
\begin{definition}\label{defn-triangulation-ug}
 Let $R$ be a self-injective ungraded ring.  An \emph{ungraded
  triangulation structure} for $R$ is a pair $(\CN,\Dl)$, where
 \begin{itemize}
  \item[(a)] $\CN$ is a full subcategory of $\InjMod_R$ containing $R$.
  \item[(b)] $\CN$ is closed under finite direct sums, retracts,
   suspensions and desuspensions.
  \item[(c)] $\Dl$ is a class of distinguished triangles
   making $(\CN,1,\Dl)$ into a triangulated category.
 \end{itemize}
\end{definition}

In~\cite{muscst:tcw} we constructed ungraded triangulation structures
for $\Z/4$ and for $K[\ep]/\ep^2$ (where $K$ is any field of
characteristic two).  If Freyd's Generating Hypothesis is true, then
the image of the functor $\pi_*$ gives a graded triangulation
structure for the ring $\pi_*(S)^\wedge_p$.  We have not succeeded in
constructing any examples of graded triangulation structures by pure
algebra.  Here we offer only some rather limited and negative results.

\begin{lemma}\label{lem-triangles-exact}
 If $(\CN,\Dl)$ is a triangulation structure (in the graded or
 ungraded context) then all distinguished triangles
 in $\Dl$ are exact sequences.
\end{lemma}
\begin{proof}
 The general theory of triangulated categories tells us that all
 functors of the form $\CN(X,-)$ send distinguished triangles to long
 exact sequences.  By assumption we have $R\in\CN$, and we can take
 $X=R$ to prove the claim.
\end{proof}

\begin{lemma}\label{lem-surjective-split}
 If $(\CN,\Dl)$ is a triangulation structure then all surjective maps
 in $\CN$ are split.  
\end{lemma}
\begin{proof}
 Let $M\xra{f}N$ be a surjective map in $\CN$.  This must fit into a
 distinguished triangle $L\xra{e}M\xra{f}N\xra{g}\Sg L$.  Here $gf=0$
 but $f$ is surjective so $g=0$.  It is standard that the functor
 $\CN(N,-)$ converts our distinguished triangle to an exact sequence,
 so $f_*\:\CN(N,M)\to\CN(N,N)$ is surjective.  We can thus find
 $h\:N\to M$ with $fh=1$, so $h$ splits $f$.
\end{proof}

\begin{corollary}
 If $(\CN,\Dl)$ is a triangulation structure then all finitely
 generated modules in $\CN$ are projective.  Thus, if $R$ is local
 then all such modules are free.
\end{corollary}
\begin{proof}
 Let $N$ be a finitely generated module in $\CN$.  This means that
 there is a surjective homomorphism $f\:F\to N$ for some finitely
 generated free module $F$.  As $\CN$ is standard we see that
 $F\in\CN$, so the lemma tells us that $N$ is a retract of $F$, so it
 is projective.  It is well-known that finitely generated projective
 modules over local rings are free.
\end{proof}

\begin{proposition}\label{prop-incoherent}
 Suppose that $R$ is a local graded ring with $R_i=0$ for $i<0$, and
 suppose that $R$ admits a triangulation structure.  Then $R$ is
 totally incoherent.
\end{proposition}
\begin{proof}
 Let $\mxi$ be the unique maximal ideal, and let $(\CN,\Dl)$ be a
 triangulation structure.  It is not hard to see that $\mxi_0$ is the
 unique maximal ideal in $R_0$, so $R_0$ is a local ring in the
 ungraded sense.

 Let $J$ be any finitely generated ideal.  We can then find a finitely
 generated free module $Q$ and an epimorphism $Q\to J$ such that
 $Q/\mxi Q\to J/\mxi J$ is an isomorphism.  We will write $g$ for the
 composite map $Q\to J\to R$, so that $J=\img(g)$.  If $J$ is finitely
 presented then $\ker(g)$ is again finitely generated, so we can find
 a finitely generated free module $P$ and a map $f\:P\to Q$ with
 $\img(f)=\ker(g)$ and $P/\mxi P\xra{\simeq}\ker(g)/\mxi\ker(g)$.
 With these minimal choices for $P$ and $Q$, it is clear that
 $P_i=Q_i=0$ when $i<0$.  Next, we can fit $g$ into a
 distinguished triangle $\Sg^{-1}R\xra{d}K\xra{i}Q\xra{g}R$.  As
 $gf=0$, we can find a lift $\tilde{f}\:P\to K$ with $i\tilde{f}=f$.
 We can combine this with $d$ to give a map $P\oplus\Sg^{-1}R\to K$, and
 a diagram chase shows that this is surjective.  Using
 Lemma~\ref{lem-surjective-split} we deduce that this map is split epi
 and that $K$ is a finitely generated free module.  It follows that
 $K_i=0$ for $i<-1$ and that $K_{-1}$ is a retract of $R_0$.  As $R_0$
 is local we must have either $K_{-1}=0$ or $K_{-1}=R_0$.  If
 $K_{-1}=0$ then $d\:\Sg^{-1}R\to K$ must be zero, which implies
 that $g\:Q\to R$ is split epi, which means that $J=R$.  If
 $K_{-1}\neq 0$ then we find that $d$ must induce a monomorphism
 $\Sg^{-1}R/\mxi\to K$, and as $R$ is local this implies that $d$
 is a split monomorphism, and thus that $g=0$ and so $J=0$.
\end{proof}
\begin{remark}
 As mentioned previously, there is an ungraded triangulation structure
 for the ring $\Z/4$.  The ideal $(2)<\Z/4$ is finitely presented and
 is neither $0$ nor $\Z/4$.  It follows that our grading assumptions
 are playing an essential role in the proof of the above proposition.
\end{remark}

\begin{corollary}
 Neither the infinite exterior algebra (as in
 Example~\ref{eg-exterior}) nor the cube algebra (as in
 Section~\ref{sec-cube}) admits a triangulation structure.
\end{corollary} 
\begin{proof}
 Both rings are coherent, by Propositions~\ref{prop-exterior-coherent}
 and~\ref{prop-cube-coherent}.  
\end{proof}

\end{document}